\documentclass[10pt,letterpaper]{amsart}
\usepackage[T1]{fontenc}
\usepackage{lmodern}
\usepackage{amsmath,amssymb,amsthm,mathtools}
\usepackage{dsfont}
\usepackage[headings]{fullpage}
\usepackage{thmtools}

\makeatletter
\renewcommand\thmt@autorefsetup{\@xa\def\csname\thmt@envname autorefname\@xa\endcsname\@xa{\thmt@thmname}}
\makeatother

\usepackage{enumitem}
\usepackage{microtype}
\usepackage{booktabs,array,graphicx}
\usepackage[dvipsnames]{xcolor}
\usepackage[hidelinks]{hyperref}
\usepackage[backend=biber,style=alphabetic,maxbibnames=999,
 maxalphanames=999,doi=true,url=false]{biblatex}
\DeclareLabelalphaTemplate{\labelelement{\field[final]{shorthand}}\labelelement{\field[uppercase,strwidth=1,strside=left]{labelname}}\labelelement{\field[strwidth=2,strside=right]{year}}}
\AtBeginDocument{}
\usepackage[nameinlink,noabbrev]{cleveref}
\hypersetup{colorlinks=true,linkcolor=blue,urlcolor=red,citecolor=Green,
 pdftitle={Counterexamples to the Reiner--Shimozono conjecture and the failure of Schubert filtrations},
 pdfauthor={Reuven Hodges},
 pdfsubject={Counterexamples, coefficient extraction, and positivity criteria},
 pdfkeywords={key polynomials, Demazure atoms, Schubert filtrations, coefficient extraction}}
\numberwithin{equation}{section}
\numberwithin{figure}{section}
\setlist[enumerate,1]{label=(\alph*),ref=\alph*}
\theoremstyle{plain}
\newtheorem{theorem}[equation]{Theorem}
\newtheorem{proposition}[equation]{Proposition}
\newtheorem{lemma}[equation]{Lemma}

\newtheorem{maintheorem}{Theorem}[section]
\newtheorem{maincorollary}[maintheorem]{Corollary}
\theoremstyle{definition}
\newtheorem{definition}[equation]{Definition}

\newtheorem{remark}[equation]{Remark}
\crefname{maintheorem}{theorem}{theorems}
\Crefname{maintheorem}{Theorem}{Theorems}
\crefname{maincorollary}{corollary}{corollaries}
\Crefname{maincorollary}{Corollary}{Corollaries}
\DeclareMathOperator{\sgn}{sgn}

\title[Counterexamples to atom positivity]{Counterexamples to the Reiner--Shimozono conjecture and the failure of Schubert filtrations}
\author{Reuven Hodges}
\address{Department of Mathematics, University of Kansas, Lawrence, KS 66045}
\email{\href{mailto:rmhodges@ku.edu}{rmhodges@ku.edu}}
\date{September 19, 2026}

\begin{document}
\raggedbottom
\begin{abstract}
We disprove the Reiner--Shimozono conjecture that products of key
polynomials have nonnegative expansions in Demazure atoms. By relaxing the
defining relations of Demazure modules, we obtain an exact
coefficient-extraction formula that yields an explicit infinite family of
counterexamples, including negative coefficients in twenty-eight variables. These examples
answer van der Kallen's long-standing open question on relative
Schubert filtrations negatively, already for tensor products of two dual
Joseph modules.
For individual atom coefficients, we give a sufficient criterion for positivity
that is checkable in polynomial time in the binary input length, uniformly
in rank. For each fixed rank $n\ge4$, this criterion recognizes a nonvanishing
proportion of the positive atom coefficients as the bound on the composition
entries tends to infinity.
\end{abstract}
\maketitle

\section{Introduction}\label{sec:introduction}

The Reiner--Shimozono conjecture concerns a striking pattern in the
multiplication of key polynomials, first observed in unpublished work of
Reiner and Shimozono from circa 1996 \cite{Pun2016,Reiner2026}.
Although products of keys can have
negative coefficients in the key basis, and products of Demazure atoms can
have negative coefficients in the atom basis, the conjecture predicts that
multiplying keys and expanding in atoms restores positivity.
This expectation finds support
in extensive computations and in theorems covering substantial infinite
families. Pun's theorem establishes positivity for every pair of keys in
three variables, without any restriction on their degrees \cite{Pun2016},
while positive multiplication rules with a Schur factor hold in arbitrarily
many variables \cite{HaglundLuotoMasonVanWilligenburg2011},
including when the Schur factor uses only an initial segment of the
variables \cite{AssafInsertion}.
These results, together with the many related results discussed below,
approach the conjecture from different directions, and their agreement
suggests that there should be a common reason for the positivity.
Indeed, this expectation has a geometric counterpart in van der Kallen's
long-standing open question whether tensor products of modules with
excellent filtrations admit relative Schubert filtrations, whose successive
quotients have
atom-positive characters and would therefore imply the Reiner--Shimozono
conjecture \cite[Section~6.3, Question~1]{VanderKallen1993}.
An explanation of why multiplication should produce positivity between
two bases neither of which has positive multiplication has
remained elusive.

In this paper, we construct an infinite family of counterexamples to the
Reiner--Shimozono conjecture. These also provide explicit counterexamples
to Polo's Schubert-filtration question and answer van der Kallen's
question on relative Schubert filtrations negatively.
Our method replaces Demazure modules by larger modules, which we call
\emph{relaxed Demazure modules}, whose simpler characters allow us to
compute individual atom coefficients under explicit hypotheses.

We first fix the conventions needed to state these counterexamples.
Fix $n\geq1$, and let the symmetric group $S_n$ act on
$\mathbb Z[x_1,\ldots,x_n]$ by permuting variables and on weak
compositions in $\mathbb Z_{\geq0}^n$ by permuting coordinates.
Write $s_i=(i,i+1)$ for the adjacent transpositions. For a weak
composition $a$, write $|a|=\sum_i a_i$ and $x^a=\prod_i x_i^{a_i}$.

The isobaric divided difference operators $\pi_i$ and their shifted
counterparts $\overline\pi_i$ are defined by
\begin{equation}\label{eq:intro-operators}
 \pi_i f=\frac{x_i f-x_{i+1}s_i f}{x_i-x_{i+1}},
 \qquad \overline\pi_i=\pi_i-1
 \qquad (1\leq i<n).
\end{equation}
Both families satisfy the braid relations. Thus, for $\sigma\in S_n$,
we may define $\pi_\sigma=\pi_{i_1}\cdots\pi_{i_\ell}$ and
$\overline\pi_\sigma=\overline\pi_{i_1}\cdots\overline\pi_{i_\ell}$
using any reduced expression $\sigma=s_{i_1}\cdots s_{i_\ell}$.

For a weak composition $a$, let $a^+$ be its weakly decreasing
rearrangement and let $\sigma\in S_n$ be the permutation of minimal
length satisfying $\sigma a^+=a$. The \emph{key polynomial} and
\emph{Demazure atom} indexed by $a$ are, respectively,
\[
 \kappa_a=\pi_\sigma x^{a^+},\qquad
 \mathcal A_a=\overline\pi_\sigma x^{a^+}.
\]
Each family forms an integral basis of $\mathbb Z[x_1,\ldots,x_n]$.
Equivalently, the key polynomials are characterized by the recursion
\begin{equation}\label{eq:intro-key-recursion}
 \kappa_\lambda=x^\lambda\quad\text{for weakly decreasing }\lambda,
 \qquad \kappa_a=\pi_i\kappa_{s_i a}\quad\text{if }a_i<a_{i+1}.
\end{equation}

The refinement of key polynomials into Demazure atoms goes back to
Lascoux and Sch\"utzenberger's \emph{Keys \& standard bases},
where the atoms are called standard bases \cite{LascouxSchutzenberger1990}.
For a weakly decreasing composition $\lambda$, let $W^\lambda$ denote
the set of minimal-length representatives of the cosets of its stabilizer
in $S_n$. Their decomposition takes the form
\begin{equation}\label{eq:intro-bru-refinement}
 \kappa_{\sigma\lambda}
 =\sum_{\substack{u\in W^\lambda\\u\leq\sigma}}\mathcal A_{u\lambda},
 \qquad \sigma\in W^\lambda,
\end{equation}
where $\leq$ denotes Bruhat order.

Key polynomials include both monomials and Schur polynomials. For a
weakly decreasing composition $\lambda$, we have $\kappa_\lambda=x^\lambda$
and $\kappa_{\omega\lambda}=s_\lambda(x_1,\ldots,x_n)$,
where $\omega\in S_n$ is the longest permutation, which reverses the
coordinates. In general, a key polynomial depends on the order of the
entries of its indexing composition, not merely on their decreasing
rearrangement. Keys also provide the expansion basis in Reiner and
Shimozono's flagged Littlewood--Richardson rule, which expresses a flagged
skew Schur polynomial as a nonnegative integral sum of key polynomials
\cite{ReinerShimozono1995}.

Write $[\mathcal A_c]f$ for the coefficient of $\mathcal A_c$ in the
Demazure-atom expansion of $f$. The Reiner--Shimozono conjecture asserts that
\begin{equation}\label{eq:intro-rs-positive}
 [\mathcal A_c](\kappa_a\kappa_b)\geq0
 \qquad\text{for every }n\geq1\text{ and }a,b,c\in\mathbb Z_{\geq0}^n.
\end{equation}
Pun records this statement as Conjecture~1 \cite{Pun2016}.
Pechenik and Searles state it as Conjecture~4.32 in their survey,
observing that ``the conjecture is quite mysterious'' \cite{PechenikSearles2020}.

Our main result evaluates an atom coefficient for an explicit family of
products and determines exactly when it is negative. In the compositions
below, superscripts denote repetition, with exponent zero contributing no
entries, and the ellipses run through successive multiples of $\delta$.

\begin{maintheorem}\label{thm:intro-family}
Let $p,q,\delta$ be positive integers with $\delta\ge p+q$.
In $4(p+q-1)$ variables, define
\begin{equation}\label{eq:intro-family-arrays}
\begin{aligned}
 a={}&(0^p,\delta^q,2\delta,3\delta,\ldots,(p+q)\delta,0^{p-1},\delta^{q-1},(p+q)\delta,(p+q-1)\delta,\ldots,2\delta),\\
 b={}&(\delta^p,0^q,(p+q)\delta,(p+q-1)\delta,\ldots,2\delta,\delta^{p-1},0^{q-1},2\delta,3\delta,\ldots,(p+q)\delta),\\
 c={}&((\delta+1)^{p+q},((p+q+2)\delta-1)^{p+q-1},(\delta+1)^{p+q-2},((p+q+2)\delta-1)^{p+q-1}).
\end{aligned}
\end{equation}
Then
\begin{equation}\label{eq:intro-family-value}
 [\mathcal A_c](\kappa_a\kappa_b)
 =\frac{p+q-1}{pq}\binom{p+q-2}{p-1}^{\!2}
    \bigl(2-(p-2)(q-2)\bigr)
 =N(p+q-1,p)\bigl(2-(p-2)(q-2)\bigr),
\end{equation}
where $N(r,k)=\frac{1}{r}\binom{r}{k}\binom{r}{k-1}$
is the Narayana number for $1\le k\le r$.
In particular, this coefficient is negative if and only if
$(p-2)(q-2)>2$.
\end{maintheorem}
\noindent Taking $(p,q)=(3,5)$ and $(4,4)$ gives coefficients $-105$ and $-350$,
respectively, in $28$ variables for every $\delta\ge8$.

\subsection{Schubert filtrations and consequences}\label{sec:consequences}

Polo asked whether tensor products of section modules on Schubert
varieties admit Schubert filtrations \cite[Question~2.9]{Polo1989}.
In 1988, he reported an unpublished counterexample communicated to him
by van der Kallen \cite{Polo1988}. Van der Kallen subsequently asked
whether a relaxed version of Polo's conjecture might nevertheless hold,
namely whether tensor products of modules with excellent filtrations
admit relative Schubert filtrations
\cite[Section~6.3, Question~1]{VanderKallen1993}.
Our counterexamples answer this question negatively and provide an
explicit infinite family of counterexamples to Polo's original question.
Reiner and Shimozono formulated a related conjecture asserting that
arbitrary flagged Schur modules admit Schubert filtrations
\cite[Conjecture~32]{ReinerShimozono1999}.
In 2005, Stembridge found a type $A_5$ tensor-product counterexample which disproves that conjecture
\cite{Reiner2026}. Stembridge's example nevertheless has an atom-positive
character, so it does not disprove the atom-positivity conjecture or supply
the character obstruction to relative Schubert filtrations used here
\cite{Reiner2026}.

We work with $G=\operatorname{GL}_n(\mathbb C)$ and its upper-triangular
Borel subgroup $B$. We refer to spaces of global sections of line bundles
on Schubert varieties in $G/B$, equipped with their natural $B$-actions,
as \emph{section modules}. We take the line bundles on the Schubert
varieties to be restrictions of globally generated line bundles on $G/B$.
A \emph{Schubert filtration} is a filtration by $B$-submodules whose
successive quotients are spaces of sections over unions of Schubert
varieties \cite{Polo1989}.

The connection with keys and atoms uses the standard identification of the
weights of the diagonal torus with integer tuples $\nu\in\mathbb Z^n$.
Let $\eta$ be the
weakly increasing rearrangement of $\nu$, and choose the shortest
permutation $\sigma$ satisfying $\sigma\eta=\nu$. In van der Kallen's
notation, the \emph{dual Joseph module} $P(\nu)$ is the space of global
sections of the line bundle with fibre weight $\eta$ on the Schubert
variety $X_\sigma=\overline{B\sigma B/B}$.
An \emph{excellent filtration} is a filtration by $B$-submodules whose
successive quotients are dual Joseph modules $P(\nu)$.
The submodule $Q(\nu)$ of $P(\nu)$, called
a \emph{minimal relative Schubert module}, consists of sections vanishing
on the Schubert boundary; this boundary is the union of the smaller
Schubert varieties contained in $X_\sigma$. A \emph{relative Schubert filtration} is a finite
filtration by $B$-submodules with successive quotients of the form $Q(\nu)$
\cite{VanderKallen1993}. With our choice of $B$, antidominant weights are
weakly increasing. We record a weight $\mu$ in the character by the
monomial $x^{-\mu}$, so that, for every weak composition $a$,
\begin{equation}\label{eq:intro-geometric-characters}
 \operatorname{ch}P(-a)=\kappa_a,
 \qquad \operatorname{ch}Q(-a)=\mathcal A_a.
\end{equation}
The geometric construction and these character identities are detailed in
Section~\ref{sec:geometric-setup}.
The connection between Schubert filtrations and atom positivity was
also noted by Assaf \cite{AssafInsertion}.

\begin{maincorollary}\label{cor:intro-filtrations}
For the compositions $a,b$ of Theorem~\ref{thm:intro-family} with
$(p-2)(q-2)>2$, the tensor product $P(-a)\otimes P(-b)$ admits neither
a relative Schubert filtration nor a Schubert filtration in Polo's sense.
These counterexamples lie in type $A_{4(p+q)-5}$, with both failures
occurring already in type $A_{27}$.
\end{maincorollary}

Each factor in Corollary~\ref{cor:intro-filtrations} has an excellent
filtration consisting of the single step $0\subset P(-a)$ or
$0\subset P(-b)$, while their tensor product has no relative Schubert
filtration.

There is also a consequence for $K$-theory. Lascoux polynomials and
Lascoux atoms are deformations of keys and Demazure atoms, respectively,
with deformation parameter $\beta$. Setting $\beta=0$ recovers the
corresponding key polynomial or Demazure atom. Monical, Pechenik, and
Searles conjectured that products of Lascoux polynomials expand in Lascoux
atoms with coefficients in $\mathbb Z_{\geq0}[\beta]$
\cite[Conjecture~4.25]{MonicalPechenikSearles2021}.
Such an expansion would specialize at $\beta=0$ to an atom-positive
expansion of the corresponding product of keys.

\begin{maincorollary}\label{cor:intro-lascoux}
For the compositions $a,b$ of Theorem~\ref{thm:intro-family} with
$(p-2)(q-2)>2$, the product of their Lascoux polynomials
has no expansion in Lascoux atoms with coefficients in
$\mathbb Z_{\geq0}[\beta]$.
\end{maincorollary}

\subsection{Coefficient extraction and positivity consequences}
\label{sec:intro-extraction}

Our coefficient-extraction method begins with Fu--Lascoux's duality between
key polynomials and Demazure atoms \cite{FuLascoux}. In
Lemma~\ref{lem:duality}, we use this duality to express an individual atom
coefficient as an ordinary Laurent coefficient. For
$f\in\mathbb Z[x_1,\ldots,x_n]$, a weak composition
$c\in\mathbb Z_{\geq0}^n$, and an integer $N\geq\max_i c_i$,
write $\mathbf1=(1,\ldots,1)\in\mathbb Z^n$.
Then
\begin{equation}\label{eq:intro-duality}
 [\mathcal A_c]f=[x^{N\mathbf1}]\,
 f\,\kappa_{N\mathbf1-c}\prod_{i<j}(1-x_i/x_j),
\end{equation}
where $[x^v]$ denotes ordinary Laurent coefficient extraction.

Identity \eqref{eq:intro-duality} expresses the atom coefficient as a
Laurent coefficient, but its right-hand side still contains key polynomials.
To evaluate this coefficient when $f=\kappa_a\kappa_b$, we replace the three
Demazure modules corresponding to these keys by larger modules with simpler
characters. We construct these \emph{relaxed Demazure modules} by omitting
the higher-power relations in the cyclic presentation of a Demazure module.
These modules were also considered by Joseph
\cite[\S2.6]{Joseph1986} in his study of the Lie algebra homology
of Demazure modules for sufficiently dominant highest weights.
Their characters are monomials times products of
geometric series. Explicit inequalities on the indexing compositions
ensure that the omitted relations cannot affect the coefficient being
extracted, allowing us to replace the three keys in
\eqref{eq:intro-duality} by these simpler characters. This gives the
\emph{rational coefficient formula} \eqref{eq:rational-extraction}, under
the hypotheses of Proposition~\ref{prop:window}. The same method yields
both our negative coefficients and the positivity results below. We expect
that suitable modifications of the relaxed Demazure module construction
will extend this approach to broader classes of positivity questions.

The positivity results concern \emph{quiver triples}, a class of triples of
weak compositions defined in Definition~\ref{def:quiver-triple}.
For a quiver triple $(a,b,c)$ of length $n$, fix an interval partition
$\mathcal I=(I_1,\ldots,I_s)$ satisfying that definition.
Let $V_p\subseteq\mathbb C^n$ be the coordinate subspace indexed by $I_p$,
and let $L_{\mathcal I}=\prod_{p=1}^s\operatorname{GL}(V_p)$ be the
corresponding standard Levi subgroup.
Section~\ref{subsec:quiver-multiplicity} constructs an acyclic quiver $Q$
with vertices $1,\ldots,s$. Its representation space is
$\bigoplus_{e:p\to q}\operatorname{Hom}(V_p,V_q)$; write $\mathcal R_Q$
for its coordinate ring, with the action induced by change of basis
at the vertices.

For $I_p=\{i_1<\cdots<i_d\}$, the tuple
$\lambda^{(p)}=(c_{i_d}-a_{i_d}-b_{i_d},\ldots,c_{i_1}-a_{i_1}-b_{i_1})$
is dominant. Let $V_p^{\lambda^{(p)}}$ denote the irreducible rational
$\operatorname{GL}(V_p)$-module with this highest weight.

\begin{maintheorem}\label{thm:intro-quiver-polytope}
For every quiver triple $(a,b,c)$, with the associated data above,
there is a bounded rational polytope $P(a,b,c)\subseteq\mathbb R^M$ such that
\[
 [\mathcal A_c](\kappa_a\kappa_b)
 =\dim\operatorname{Hom}_{L_{\mathcal I}}
 \left(\bigotimes_{p=1}^s V_p^{\lambda^{(p)}},\mathcal R_Q\right)
 =\#\bigl(P(a,b,c)\cap\mathbb Z^M\bigr).
\]
Thus the atom coefficient is an irreducible
multiplicity for $L_{\mathcal I}$. Moreover,
\[
 [\mathcal A_c](\kappa_a\kappa_b)>0
 \quad\Longleftrightarrow\quad P(a,b,c)\ne\varnothing.
\]
There are deterministic algorithms that recognize quiver triples and,
for such triples, construct $P(a,b,c)$ and decide whether
$[\mathcal A_c](\kappa_a\kappa_b)>0$. Each algorithm runs in time
polynomial in the total binary encoding length of $(a,b,c)$, with the
number of variables included as part of the input.
\end{maintheorem}

\begin{maincorollary}\label{cor:intro-quiver-density}
Fix $n\geq4$. For each positive integer $H$, let $\mathcal Q_n^+(H)$ be
the set of quiver triples $a,b,c\in\{0,\ldots,H\}^n$ satisfying
$[\mathcal A_c](\kappa_a\kappa_b)>0$. Then
\[
 \#\mathcal Q_n^+(H)=\Theta_n(H^{3n-1})
 \qquad\text{as }H\to\infty.
\]
In particular, these triples constitute a proportion bounded away from
zero among all triples $a,b,c\in\{0,\ldots,H\}^n$ satisfying
$|a|+|b|=|c|$.
\end{maincorollary}

\subsection{Related work on positivity}

The Reiner--Shimozono conjecture has been the subject of sustained study,
with positivity established for extensive infinite families and new
approaches continuing to emerge. Pun proves the conjecture for every pair
of keys in three variables, without a degree bound \cite{Pun2016}, and
also proves that multiplying a Demazure atom by a monomial with weakly
decreasing exponent vector preserves atom positivity \cite[Theorem~3.13]{Pun2016}.

Products with a Schur factor have received particularly extensive
attention. Haglund, Luoto, Mason, and van Willigenburg give
Littlewood--Richardson rules expanding an atom times a Schur polynomial
into atoms and a key times a Schur polynomial into keys
\cite{HaglundLuotoMasonVanWilligenburg2011}. Assaf and Quijada develop a Pieri rule allowing the Schur factor
to involve only an initial segment of the variables
\cite{AssafQuijada2019}.
Assaf's insertion algorithm treats arbitrary Schur factors in such an
initial segment, producing positive expansions in characters of Schubert
unions and hence in atoms \cite{AssafInsertion}.
Miller subsequently gives a vertex-model
multiplication rule for the larger family of permuted-basement Demazure
atoms, which includes both ordinary keys and atoms \cite{Miller2025}.

A further collection of positive cases comes from the representation
theory of Demazure crystals, whose characters are key polynomials.
Kouno characterizes when a tensor product of two Demazure crystals
decomposes into Demazure crystals, yielding a positive key expansion and
therefore a positive atom expansion \cite{Kouno2020}.
Assaf, Dranowski, and Gonz\'alez give an equivalent local criterion for
this decomposition \cite{AssafDranowskiGonzalez2024}.
These results identify circumstances in which positivity follows from a
decomposition of the underlying crystals, providing a structural
explanation alongside the explicit multiplication rules.

The search for such explanations extends to unions of Demazure
crystals and other polynomial families. Armon studies atom positivity
for subsets of highest-weight crystals and describes the decomposition
of Demazure unions into atoms \cite{Armon2025}.
Assaf and Gonz\'alez give a local characterization of these unions and
prove that they admit disjoint decompositions into Demazure atoms,
connecting this combinatorics with relative Schubert filtrations
\cite{AssafGonzalez2025}. In another direction, Blasiak, Haiman, Morse,
Pun, and Seelinger formulate atom-positivity conjectures for flagged LLT
polynomials in their work on nonsymmetric Macdonald polynomials
\cite{BlasiakHaimanMorsePunSeelinger2025}.
The Reiner--Shimozono conjecture forms part of ongoing efforts to
understand when positive expansions in keys and atoms exist, how to
compute them, and what structures account for them.

The paper is organized as follows. Section~\ref{sec:extraction}
derives the coefficient-extraction formula. Section~\ref{sec:hall-extraction}
expresses it for Hall triples using Hall polynomials.
Section~\ref{sec:negative} evaluates the resulting coefficient by a
signed count to prove
Theorem~\ref{thm:intro-family} and the two corollaries.
Section~\ref{sec:quiver} proves the quiver formula and its density
consequence.

\section*{Acknowledgements}

The author thanks Jonathan Gruber for pointing out a remark in
Polo's 1988 paper \cite{Polo1988}, which announces van der Kallen's
counterexample to Polo's Schubert-filtration question, and Victor Reiner
for clarifying the history of
the Reiner--Shimozono conjecture and sharing details of John Stembridge's
counterexample to Conjecture~32 in \cite{ReinerShimozono1999}.
The author used OpenAI's GPT-5.6 Sol and GPT-6 Astra models on an initial draft of this 
manuscript to assist
with checking for mathematical errors, gaps in arguments,
and expository issues.
In particular, GPT-6 Astra suggested the connection between the atom
coefficients considered in Theorem~\ref{thm:intro-quiver-polytope}
and lattice-point counts in rational polytopes, and suggested the
density argument used in the proof of
Corollary~\ref{cor:intro-quiver-density}.
The author also used GPT-6 Astra to produce a formalization of
Theorem~\ref{thm:intro-family} in Lean~4~\cite{deMouraUllrich2021}, using the
Mathlib library~\cite{mathlib2020}.
The author wrote the arguments presented in this paper and takes
full responsibility for its mathematical content.

\section{Exact coefficient extraction}
\label{sec:extraction}

This section develops the coefficient-extraction formulas that underlie
our counterexamples and quiver positivity results. We first express an
atom-basis coefficient as an ordinary Laurent coefficient, then derive
the rational coefficient formula used throughout the paper.

All compositions have nonnegative integer entries. For a polynomial
$f$, the notation $[\mathcal A_c]f$ denotes the coefficient of
$\mathcal A_c$ in its atom expansion. Representation-theoretic
arguments are over $\mathbb C$; all resulting character and coefficient
identities are integral.

\subsection{Duality and the positive-root window}

Write $\mathbf1=(1,\ldots,1)$, and define the normalized type $A$
Weyl denominator by
\begin{equation}
\label{eq:weyl-factor}
 \Delta_n(x)=\prod_{1\le i<j\le n}(1-x_i/x_j).
\end{equation}
Let $\omega$ reverse coordinates, and set
$(Jf)(x_1,\ldots,x_n)=f(x_n^{-1},\ldots,x_1^{-1})$.
Fu--Lascoux's orthogonality theorem states that keys and atoms are dual,
with their indices reversed, under the pairing
\begin{equation}\label{eq:fl-pairing-definition}
 \langle f,g\rangle=[x^0]\bigl(f\,(Jg)\,\Delta_n\bigr),
\end{equation}
where $[x^0]$ denotes the constant term in all $n$ variables.
More precisely, for weak compositions $u,v$,
\begin{equation}\label{eq:fl-pairing}
 \langle\kappa_v,\mathcal A_u\rangle=\delta_{\omega v,u},
\end{equation}
where $\delta$ is the Kronecker delta
\cite[Definition~7 and Theorem~15]{FuLascoux}.
The duality \eqref{eq:fl-pairing} gives the following
coefficient-extraction formula.

\begin{lemma}\label{lem:duality}
For every polynomial $f\in\mathbb Z[x_1,\ldots,x_n]$, every weak
composition $c$ with at most $n$ parts (padded with trailing zeros
to length $n$), and every integer $N\ge\max_i c_i$,
\begin{equation}
\label{eq:duality}
 [\mathcal A_c]f=[x^{N\mathbf1}]\,
 f\,\kappa_{N\mathbf1-c}\,\Delta_n(x).
\end{equation}
\end{lemma}

\begin{proof}
The pairing \eqref{eq:fl-pairing-definition} is symmetric.
Indeed, $J$ is an involution, $J\Delta_n=\Delta_n$, and constant
term is invariant under $J$, so
\[
 \langle f,g\rangle
 =[x^0]\bigl(f\,(Jg)\,\Delta_n\bigr)
 =[x^0]J\bigl(f\,(Jg)\,\Delta_n\bigr)
 =[x^0]\bigl((Jf)\,g\,\Delta_n\bigr)
 =\langle g,f\rangle.
\]
Expanding $f$ in atoms and applying Fu--Lascoux's identity
\eqref{eq:fl-pairing} therefore gives
\begin{equation}\label{eq:atom-pair-extraction}
 [\mathcal A_c]f
 =\langle f,\kappa_{\omega c}\rangle
 =[x^0]\bigl(f\,J\kappa_{\omega c}\,\Delta_n\bigr).
\end{equation}

Let $\lambda$ be the decreasing rearrangement of $c$, and choose a
shortest sequence of adjacent swaps $s_{i_1},\ldots,s_{i_\ell}$
taking $\lambda$ to $\omega c$. Then
\begin{equation}\label{eq:key-complement}
 x^{N\mathbf1}J\kappa_{\omega c}
 =x^{N\mathbf1}J\bigl(\pi_{i_\ell}\cdots\pi_{i_1}x^\lambda\bigr)
 =\pi_{n-i_\ell}\cdots\pi_{n-i_1}\bigl(x^{N\mathbf1}Jx^\lambda\bigr)
 =\pi_{n-i_\ell}\cdots\pi_{n-i_1}x^{N\mathbf1-\omega\lambda}
 =\kappa_{N\mathbf1-c}.
\end{equation}
The first equality in \eqref{eq:key-complement} is the defining
construction of the key. The second uses $J\pi_i=\pi_{n-i}J$,
which follows from \eqref{eq:intro-operators}, and the fact that
multiplication by $x^{N\mathbf1}$ commutes with every $\pi_i$.
The third uses $Jx^\lambda=x^{-\omega\lambda}$. The fourth follows
from the defining recursion for keys \eqref{eq:intro-key-recursion},
because $a\mapsto N\mathbf1-\omega a$ carries the original
strict-descent swaps to strict-descent swaps with indices $n-i$,
starting at the decreasing composition $N\mathbf1-\omega\lambda$
and ending at $N\mathbf1-c$; the hypothesis $N\geq\max_i c_i$
ensures that all these compositions have nonnegative entries.

Substituting \eqref{eq:key-complement} into
\eqref{eq:atom-pair-extraction} yields
\[
 [\mathcal A_c]f
 =[x^0]\bigl(x^{-N\mathbf1}f\,\kappa_{N\mathbf1-c}\,\Delta_n\bigr)
 =[x^{N\mathbf1}]\bigl(f\,\kappa_{N\mathbf1-c}\,\Delta_n\bigr).
 \qedhere
\]
\end{proof}

Fix weak compositions $a,b,c\in\mathbb Z_{\ge0}^n$ with
$|a|+|b|=|c|$, and choose an integer $N\ge\max_i c_i$.
Write $\bar c=N\mathbf1-c$. Applying Lemma~\ref{lem:duality} to
$f=\kappa_a\kappa_b$ expresses the atom coefficient in terms of the
three keys indexed by $a,b,\bar c$. The difference between the target
exponent $N\mathbf1$ and the sum of these three compositions is the
\emph{residual weight}
\begin{equation}\label{eq:residual}
 N\mathbf1-a-b-\bar c=c-a-b.
\end{equation}
Define its \emph{prefix heights} by
\begin{equation}\label{eq:prefix-heights}
 h_k=\sum_{i=1}^k(c_i-a_i-b_i)\qquad(0\le k\le n).
\end{equation}
In particular, $h_0=h_n=0$.

If $\varepsilon_1,\ldots,\varepsilon_n$ are the standard coordinate
vectors, then
$c-a-b=\sum_{k=1}^{n-1}h_k(\varepsilon_k-\varepsilon_{k+1})$.
Thus the prefix heights are the coordinates of $c-a-b$ in the basis
of simple roots.

For $P,Q\in\mathbb Z^3$, define their \emph{comparison weight}
\begin{equation}\label{eq:root-weight}
 \operatorname{cmp}(P,Q)
 =\#\{\ell\in\{1,2,3\}:P_\ell<Q_\ell\}-1.
\end{equation}
This weight takes values in $\{-1,0,1,2\}$.
Let $\mathds{1}_E$ denote the indicator of a condition $E$, equal to $1$
if $E$ holds and $0$ otherwise. Then
\begin{equation}\label{eq:comparison-indicators}
 \operatorname{cmp}(P,Q)
 =\mathds{1}_{P_1<Q_1}+\mathds{1}_{P_2<Q_2}+\mathds{1}_{P_3<Q_3}-1.
\end{equation}

We say that the \emph{window inequalities} hold if each composition
$u\in\{a,b,\bar c\}$ satisfies
\begin{equation}\label{eq:window}
 u_j-u_i+1>\min_{i\le k<j}h_k
 \quad\text{whenever }1\le i<j\le n\text{ and }u_i<u_j.
\end{equation}

The window inequalities \eqref{eq:window} and the coefficient formula
\eqref{eq:rational-extraction} do not depend on
$N\ge\max_i c_i$, since changing $N$ shifts all entries of $\bar c$
by the same amount, preserving their differences and comparisons.

\begin{proposition}\label{prop:window}
Let $a,b,c\in\mathbb Z_{\ge0}^n$ satisfy $|a|+|b|=|c|$, choose
$N\ge\max_i c_i$, and put $\bar c=N\mathbf1-c$.
Define $h_k$ by \eqref{eq:prefix-heights}. If every $h_k\ge0$ and the window
inequalities \eqref{eq:window} hold, then
\begin{equation}
\label{eq:rational-extraction}
 [\mathcal A_c](\kappa_a\kappa_b)
 =[x^{c-a-b}]\prod_{i<j}
 (1-x_i/x_j)^{-\operatorname{cmp}((a_i,b_i,\bar c_i),(a_j,b_j,\bar c_j))}.
\end{equation}
Here each factor $(1-x_i/x_j)^{-1}$ is expanded as
$\sum_{r\ge0}(x_i/x_j)^r$.
\end{proposition}

\begin{proof}
Applying \eqref{eq:duality} to $f=\kappa_a\kappa_b$ expresses the
desired atom coefficient in terms of the three keys
$\kappa_a,\kappa_b,\kappa_{\bar c}$. We will use presentations of their
Demazure modules to replace these keys by explicit rational expressions
without changing the extracted coefficient.

Let $\mathfrak n^+$ be the Lie algebra of strictly upper-triangular
matrices. Its basis consists
of the matrix units $E_{ij}$, for $i<j$, where $E_{ij}$ has a $1$ in
position $(i,j)$ and zeros elsewhere. Under conjugation by the diagonal
torus, $E_{ij}$ has weight $\varepsilon_i-\varepsilon_j$; its action on
a representation therefore raises weights by this positive root.

Fix a weak composition $u$, and let $V(u^+)$ be the irreducible
$\operatorname{GL}_n$-module of highest weight $u^+$. The weight $u$,
being a permutation of $u^+$, is an extremal weight; choose a nonzero
vector $v_u$ of this weight. The span of all vectors obtained by
repeatedly applying elements of $\mathfrak n^+$ to $v_u$ is the Demazure
module with character $\kappa_u$. Here we record each weight $\eta$ by
$x^\eta$, so applying $E_{ij}$ shifts the corresponding monomial by
$x_i/x_j$. This differs from the convention $x^{-\eta}$ used for the
geometric section modules in Section~\ref{sec:consequences}.

Let $U(\mathfrak n^+)$ denote the universal enveloping algebra, the
associative algebra generated by $\mathfrak n^+$ subject to the
relations $XY-YX=[X,Y]$. Its elements act by linear combinations of
successive applications of Lie algebra operators. Thus the Demazure
module is the cyclic module $U(\mathfrak n^+)v_u$, generated by the
single vector $v_u$. Joseph's presentation identifies the kernel of
the surjection $D\mapsto Dv_u$ from $U(\mathfrak n^+)$ onto this
module as the left ideal
\[
 \sum_{i<j}U(\mathfrak n^+)\,
 E_{ij}^{\max(u_j-u_i,0)+1}.
\]
Equivalently, the relations
\begin{equation}
\label{eq:demazure-relations}
 E_{ij}^{\max(u_j-u_i,0)+1}v_u=0\qquad(i<j)
\end{equation}
generate all relations on $v_u$, with the Lie algebra commutation
rules already incorporated in $U(\mathfrak n^+)$
\cite[Theorem~3.4]{Joseph1985}; see also
\cite[Proposition~2.1]{Polo1989}.

In \eqref{eq:demazure-relations}, the exponent is $1$ precisely when
$u_i\ge u_j$. We retain these relations $E_{ij}v_u=0$ and omit the
relations $E_{ij}^{u_j-u_i+1}v_u=0$ for which $u_i<u_j$. More precisely,
we define the \emph{relaxed Demazure module associated with $u$} by
\[
 \widetilde D_u
 =U(\mathfrak n^+)\Big/
 \sum_{\substack{i<j\\u_i\ge u_j}}U(\mathfrak n^+)E_{ij}.
\]
The denominator is the left ideal generated by the retained relations.
It is contained in the defining ideal of the Demazure module, so the
map $D\mapsto Dv_u$ descends to a surjection from $\widetilde D_u$
onto that module. Imposing the omitted higher-power relations recovers
the Demazure module. The relaxed module is generally infinite
dimensional, as the PBW basis below will show.

Let $\mathfrak h_u\subseteq\mathfrak n^+$ be the span of the matrix
units $E_{ij}$ with $i<j$ and $u_i\ge u_j$. To check closure under
brackets, it suffices by bilinearity to consider two such matrix units.
A nonzero bracket arises only from a chain $i<j<k$, giving
$[E_{ij},E_{jk}]=E_{ik}$, or its negative if the order is reversed.
Since the two units belong to the chosen span, we have
$u_i\ge u_j\ge u_k$, so $E_{ik}$ also belongs to it.
Hence $\mathfrak h_u$ is closed under brackets and is therefore a Lie subalgebra.
The defining left ideal of $\widetilde D_u$ is
$U(\mathfrak n^+)\mathfrak h_u$, the left ideal generated by
$\mathfrak h_u$. The enveloping algebra acts on this quotient by left
multiplication, so every class $[D]$ is obtained by applying $D$ to $[1]$.
Thus $[1]$ generates the module; we denote it by $v_u$ and define the
compatible torus action so that it has weight $u$, matching the original
extremal vector.

To calculate the character of $\widetilde D_u$, order the matrix-unit basis of
$\mathfrak n^+$ so that the units $E_{ij}$ with $u_i<u_j$ come first
and those in $\mathfrak h_u$ come last. The
Poincar\'e--Birkhoff--Witt theorem \cite[\S17.3]{Humphreys1972}
states that, for any ordering of a basis of $\mathfrak n^+$, the
corresponding ordered monomials form a vector-space basis of
$U(\mathfrak n^+)$. Since
$\mathfrak h_u$ is a Lie subalgebra, the left ideal
$U(\mathfrak n^+)\mathfrak h_u$ is spanned by precisely those ordered
monomials containing at least one factor from $\mathfrak h_u$.
Consequently, the classes of the ordered monomials using only the
units $E_{ij}$ with $u_i<u_j$ form a basis of $\widetilde D_u$.

The basis vectors of $\widetilde D_u$ are indexed by arbitrary choices
of nonnegative integers $r_{ij}$, one for each pair $i<j$ with $u_i<u_j$.
The vector obtained by applying the corresponding ordered monomial to
$v_u$ has weight monomial
$x^u\prod_{i<j,\;u_i<u_j}(x_i/x_j)^{r_{ij}}$.
Summing these monomials over the basis gives the character; since the
exponents can be chosen independently, the sum factors into geometric series.
\begin{equation}
\label{eq:linear-character}
 \operatorname{ch}\widetilde D_u=
 \sum_{\substack{r_{ij}\ge0\\i<j,\ u_i<u_j}}
 x^u\prod_{\substack{i<j\\u_i<u_j}}(x_i/x_j)^{r_{ij}}
 =x^u\prod_{\substack{i<j\\u_i<u_j}}
 \left(\sum_{r\ge0}(x_i/x_j)^r\right)
 =x^u\prod_{\substack{i<j\\u_i<u_j}}(1-x_i/x_j)^{-1}.
\end{equation}

It remains to show that replacing
$\kappa_a,\kappa_b,\kappa_{\bar c}$ by
$\operatorname{ch}\widetilde D_a,\operatorname{ch}\widetilde D_b,
\operatorname{ch}\widetilde D_{\bar c}$, respectively, leaves the
coefficient of $x^{N\mathbf1}$ in their product with $\Delta_n$
unchanged. We will show that the vectors removed when passing from
the relaxed modules to the Demazure modules cannot contribute to
this coefficient.

A contribution to the product of the three characters comes from
choosing a weight vector in each module and adding their weights.
The three cyclic generators have combined weight $a+b+\bar c$.
Since $\bar c=N\mathbf1-c$, the definition of $h_k$ in
\eqref{eq:prefix-heights} gives
\[
 \left(\sum_{\ell=1}^{k}(a_\ell+b_\ell+\bar c_\ell)\right)+h_k=kN.
\]
Thus, to reach the target weight $N\mathbf1$, the sum of the first
$k$ weight coordinates must increase by exactly $h_k$. Applying a
matrix unit $E_{ij}$, with $i<j$, adds $1$ to coordinate $i$ and
subtracts $1$ from coordinate $j$. It therefore increases this prefix
sum by $1$ when $i\le k<j$, and leaves it unchanged otherwise.
In particular, further applications of matrix units can never
decrease any prefix sum.

Fix $u\in\{a,b,\bar c\}$. To recover the Demazure module from
$\widetilde D_u$, we impose the omitted relations
$E_{ij}^{d}v_u=0$, where $d=u_j-u_i+1$, $i<j$, and $u_i<u_j$.
Before imposing such a relation, the vector $E_{ij}^{d}v_u$ has
weight $u+d\varepsilon_i-d\varepsilon_j$. We must show that the
submodule it generates cannot contribute to the coefficient being
extracted. First choose this vector in $\widetilde D_u$ and the cyclic
generators in the other two modules. Their combined weight is
$a+b+\bar c+d\varepsilon_i-d\varepsilon_j$.
The window inequalities \eqref{eq:window} provide a $k$ with
$i\le k<j$ for which $d>h_k$. At this $k$, the sum of the first $k$
coordinates of the combined weight is
\[
 \left(\sum_{\ell=1}^{k}(a_\ell+b_\ell+\bar c_\ell)\right)+d
 =kN-h_k+d
 >kN,
\]
which exceeds the corresponding prefix sum of the target weight.

The submodule generated by $E_{ij}^{d}v_u$ is spanned by vectors
obtained by further applications of matrix units. The other two
modules are likewise spanned by vectors obtained by applying matrix
units to their cyclic generators. None of these applications can
decrease the prefix sum, so every combined weight obtained in this
way still has its first $k$ coordinates summing to more than $kN$.
Finally, every monomial occurring in
$\Delta_n=\prod_{i<j}(1-x_i/x_j)$ is a product of factors $x_i/x_j$
with $i<j$, each of which also contributes a nonnegative amount to
every prefix sum. Consequently, no contribution involving this
relation submodule can have weight $N\mathbf1$, even after
multiplication by $\Delta_n$.

The omitted relations generate the kernel of the surjection from
each relaxed module to its Demazure module. Quotienting by these
kernels therefore leaves the coefficient of $x^{N\mathbf1}$ in the
product of the three characters with $\Delta_n$ unchanged.

Applying \eqref{eq:duality} and replacing each key by the character
\eqref{eq:linear-character} of its relaxed Demazure module
gives
\[
\begin{aligned}
 [\mathcal A_c](\kappa_a\kappa_b)
 &=[x^{N\mathbf1}]\kappa_a\kappa_b\kappa_{\bar c}\,\Delta_n\\
 &=[x^{N\mathbf1}]
 \frac{x^{a+b+\bar c}\Delta_n}
 {\displaystyle
  \prod_{\substack{i<j\\a_i<a_j}}(1-x_i/x_j)
  \prod_{\substack{i<j\\b_i<b_j}}(1-x_i/x_j)
  \prod_{\substack{i<j\\\bar c_i<\bar c_j}}(1-x_i/x_j)}\\
 &=[x^{c-a-b}]\prod_{i<j}
 (1-x_i/x_j)^{1-\mathds{1}_{a_i<a_j}-\mathds{1}_{b_i<b_j}
                   -\mathds{1}_{\bar c_i<\bar c_j}}\\
 &=[x^{c-a-b}]\prod_{i<j}
 (1-x_i/x_j)^{-\operatorname{cmp}((a_i,b_i,\bar c_i),(a_j,b_j,\bar c_j))}.
\end{aligned}
\]
These coefficient extractions involve only finitely many terms.
Indeed, each occurrence of $x_i/x_j$ contributes $1$ to every prefix
sum with $i\le k<j$. Since all such contributions are nonnegative,
a term contributing to $x^{c-a-b}$ can contain at most
$\min_{i\le k<j}h_k$ occurrences of $x_i/x_j$.
For the penultimate equality, removing the initial monomial leaves the
target exponent $N\mathbf1-a-b-\bar c=c-a-b$. Each pair $i<j$
contributes one numerator factor from $\Delta_n$ in
\eqref{eq:weyl-factor} and one denominator factor for each strict
increase among $a_i<a_j$, $b_i<b_j$, and $\bar c_i<\bar c_j$.
The last equality uses \eqref{eq:comparison-indicators}, proving
\eqref{eq:rational-extraction}.
\end{proof}

\section{Hall polynomial extraction}\label{sec:hall-extraction}

This section simplifies the rational coefficient formula
\eqref{eq:rational-extraction} for the compositions used in our
counterexamples. We perform part of the extraction explicitly,
obtaining polynomials whose monomials are products of distinct
variables satisfying specified restrictions on their indices.
This reduces the desired atom coefficient to a finite signed count,
which we evaluate in Section~\ref{sec:negative}.

\subsection{A coefficient-extraction identity}

Let $d,m$ be nonnegative integers, let
$x_1,\ldots,x_d,t_1,\ldots,t_m$ be independent variables, and fix
integers $0\le\ell_1\le\cdots\le\ell_m\le d$.
Write $x=(x_1,\ldots,x_d)$, so that
$\Delta_d(x)=\prod_{1\le i<j\le d}(1-x_i/x_j)$.

A subset $\{j_1<\cdots<j_d\}\subseteq\{1,\ldots,m\}$ is
\emph{Hall admissible} if $\ell_{j_r}\ge r$ for every $1\le r\le d$.
Associate to this subset the monomial $t_{j_1}\cdots t_{j_d}$.

For a finite set $R$ of variables, let $h_q(R)$ denote the complete
homogeneous symmetric polynomial of degree $q$, the sum of all
monomials of total degree $q$ in those variables. We use the
conventions $h_0(R)=1$ and $h_q(R)=0$ for $q<0$.

For a partition $\lambda$ and a weakly increasing sequence of
positive integers $b_1,\ldots,b_{\ell(\lambda)}$, a
\emph{flagged semistandard tableau} of shape $\lambda$ is a filling
of the diagram of $\lambda$ by positive integers, weakly increasing
along rows and strictly increasing down columns, with every entry
in row $i$ at most $b_i$. The \emph{weight} of a tableau $T$ is the
monomial $y^T=\prod_{r\ge1}y_r^{m_r(T)}$, where $m_r(T)$ is the
number of entries equal to $r$. The \emph{flagged Schur polynomial}
$s_\lambda(y;b)$ is the sum of the weights of all flagged
semistandard tableaux of shape $\lambda$ with flag
$b=(b_1,\ldots,b_{\ell(\lambda)})$.

\begin{lemma}\label{lem:hall-determinant}
Put $R_i=\{t_j:\ell_j\ge i\}$. Then
\begin{equation}
\label{eq:hall-determinant}
\begin{split}
 &[x_1\cdots x_d]\,
 \Delta_d(x)\prod_{j=1}^m\prod_{i=1}^{\ell_j}(1-x_i t_j)^{-1}\\
 &\hspace{8mm}=\det\bigl(h_{1+i-j}(R_i)\bigr)_{i,j=1}^d
 =\sum_{\substack{1\le j_1<\cdots<j_d\le m\\\ell_{j_r}\ge r\ (1\le r\le d)}}
 t_{j_1}\cdots t_{j_d}.
\end{split}
\end{equation}
The sum on the right ranges over the Hall-admissible subsets,
with each subset contributing its associated monomial.
For $d=0$ all three expressions are $1$.
\end{lemma}

\begin{proof}
The case $d=0$ is immediate, so assume $d\ge1$.
For each $i$, expand every factor $(1-x_it_j)^{-1}$ with
$\ell_j\ge i$ as the geometric series $1+x_it_j+x_i^2t_j^2+\cdots$.
Multiplying these series means choosing one term $(x_it_j)^{a_j}$
from each factor, with $a_j\ge0$. The resulting product has
exponent $\sum_j a_j$ on $x_i$ and coefficient $\prod_j t_j^{a_j}$,
where both indices range over $j$ with $\ell_j\ge i$.
The choices with $\sum_j a_j=q$ give every degree-$q$ monomial
in $R_i$ exactly once. Summing these contributions gives
\begin{equation}\label{eq:hall-row-series}
 \prod_{\substack{1\le j\le m\\\ell_j\ge i}}(1-x_it_j)^{-1}
 =\sum_{q\ge0}h_q(R_i)x_i^q.
\end{equation}

Each factor $(1-x_it_j)^{-1}$ involves exactly one of the variables
$x_1,\ldots,x_d$. Grouping the factors according to that variable
and applying \eqref{eq:hall-row-series} gives
\begin{equation}\label{eq:hall-product-series}
\begin{aligned}
 \prod_{j=1}^m\prod_{i=1}^{\ell_j}(1-x_it_j)^{-1}
 &=\prod_{i=1}^d
 \left(\sum_{q_i\ge0}h_{q_i}(R_i)x_i^{q_i}\right)\\
 &=\sum_{q_1,\ldots,q_d\ge0}
 \left(\prod_{i=1}^d h_{q_i}(R_i)\right)
 x_1^{q_1}\cdots x_d^{q_d}.
\end{aligned}
\end{equation}
Thus the coefficient of $x_1^{q_1}\cdots x_d^{q_d}$ in the double
product is $\prod_i h_{q_i}(R_i)$.

Reversing the columns in the Vandermonde determinant formula
\cite[eq.~(7.55)]{Stanley1999} and dividing by
$\prod_{i=1}^d x_i^{i-1}$ gives
$\Delta_d(x)=\det(x_i^{j-i})_{i,j=1}^d$.
Expanding this determinant gives
$\Delta_d(x)=\sum_{\sigma\in S_d}\sgn(\sigma)
\prod_{i=1}^d x_i^{\sigma(i)-i}$. Fix $\sigma\in S_d$.
Multiplying the summand indexed by $\sigma$ in the Vandermonde
determinant identity by the term indexed by $(q_1,\ldots,q_d)$
in \eqref{eq:hall-product-series} gives exponent
$\sigma(i)-i+q_i$ on $x_i$. To contribute to the coefficient of
$x_1\cdots x_d$, these exponents must all equal $1$.
Consequently, $q_i=1+i-\sigma(i)$ for every $i$, and the
contribution is
\[
 \sgn(\sigma)
 \prod_{i=1}^d h_{1+i-\sigma(i)}(R_i).
\]
If any required $q_i$ is negative, there is no such term, in
agreement with our convention $h_q=0$ for $q<0$.
Summing over $\sigma$ gives the determinant expansion of
$\det(h_{1+i-j}(R_i))_{i,j=1}^d$, proving the first equality.

For the second equality, put $y_j=t_{m+1-j}$ and
$b_i=\#\{j:\ell_j\ge d+1-i\}$, so that $b_1\le\cdots\le b_d$.
Since the $\ell_j$ are weakly increasing, we have
$R_{d+1-i}=\{y_1,\ldots,y_{b_i}\}$.
Reverse the rows and columns of the matrix
$\bigl(h_{1+i-j}(R_i)\bigr)_{i,j=1}^d$.
Each reversal multiplies its determinant by $(-1)^{\binom d2}$,
so the two signs cancel. The entry in position $(i,j)$ becomes
$h_{1-i+j}(R_{d+1-i})$. Using
$R_{d+1-i}=\{y_1,\ldots,y_{b_i}\}$, we obtain
\begin{equation}\label{eq:hall-reversed-determinant}
 \det\bigl(h_{1+i-j}(R_i)\bigr)_{i,j=1}^d
 =\det\bigl(h_{1-i+j}(y_1,\ldots,y_{b_i})\bigr)_{i,j=1}^d.
\end{equation}
If $b_1=0$, the first row of this matrix is zero.
Moreover, no $\ell_j$ equals $d$, so no choice of
$j_1<\cdots<j_d$ satisfies the final condition $\ell_{j_d}\ge d$
in the sum in \eqref{eq:hall-determinant}.
Thus both the determinant and the sum vanish.
Henceforth assume $b_1\ge1$.

Gessel's flagged Jacobi--Trudi identity
\cite[Theorem~1.3]{Wachs1985}, applied with $\lambda=(1^d)$,
$y=(y_1,\ldots,y_m)$, and $b=(b_1,\ldots,b_d)$, gives the first equality in
\begin{equation}\label{eq:flagged-column}
 \det\bigl(h_{1-i+j}(y_1,\ldots,y_{b_i})\bigr)_{i,j=1}^d
 =s_{(1^d)}(y;b)
 =\sum_{\substack{1\le r_1<\cdots<r_d\le m\\r_i\le b_i\ (1\le i\le d)}}
 y_{r_1}\cdots y_{r_d}.
\end{equation}
The second equality holds because a flagged semistandard tableau
of shape $(1^d)$ is a single column with entries $r_1<\cdots<r_d$
satisfying $r_i\le b_i$.

It remains to identify the sum in \eqref{eq:flagged-column} with
the sum in \eqref{eq:hall-determinant}. For each sequence
$r_1<\cdots<r_d$ occurring in \eqref{eq:flagged-column}, set
$j_k=m+1-r_{d+1-k}$. Then $j_1<\cdots<j_d$, and the bounds
$r_i\le b_i$ are equivalent to $\ell_{j_k}\ge k$.
Since $y_r=t_{m+1-r}$, the monomial $y_{r_1}\cdots y_{r_d}$
becomes $t_{j_1}\cdots t_{j_d}$. This correspondence is reversible,
so the two sums agree, proving the second equality in
\eqref{eq:hall-determinant}.
\end{proof}

We now describe the Hall-admissible subsets using strictly
increasing lower bounds on their elements. For the first element,
the smallest permitted index is $\min\{j:\ell_j\ge1\}$.
Each subsequent element must satisfy $\ell_{j_k}\ge k$ and be
strictly larger than its predecessor. Accordingly, set $f_0=0$
and define
\begin{equation}\label{eq:canonical-flags}
 f_k=\max\bigl(\min\{j:\ell_j\ge k\},\,f_{k-1}+1\bigr)
 \qquad(1\le k\le d),
\end{equation}
where each minimum ranges over $1\le j\le m$ and an empty
minimum is $m+1$.

\begin{lemma}\label{lem:hall-flags}
A subset $\{j_1<\cdots<j_d\}\subseteq\{1,\ldots,m\}$ is
Hall admissible if and only if $j_k\ge f_k$ for every $k$.
\end{lemma}

\begin{proof}
For $d=0$, both conditions are vacuous. If the subset is Hall
admissible, then $j_k\ge\min\{j:\ell_j\ge k\}$ for every $k$.
In particular, $j_1\ge f_1$. Together with
$j_k\ge j_{k-1}+1$ for $k\ge2$, induction using
\eqref{eq:canonical-flags} gives $j_k\ge f_k$ for every $k$.
Conversely, if $j_k\ge f_k$, then
$j_k\ge\min\{j:\ell_j\ge k\}$. This minimum cannot be empty,
since that would force $j_k\ge m+1$. As the $\ell_j$ are weakly
increasing, we obtain $\ell_{j_k}\ge k$ for every $k$.
Thus the subset is Hall admissible.
\end{proof}

\subsection{Coefficient extraction for Hall triples}

\begin{definition}\label{def:paired-target}
Let $n\ge1$ and $1\le r\le n$ be integers, and let
$a,b,c\in\mathbb Z_{\geq0}^n$ satisfy $|a|+|b|=|c|$.
Choose $N\geq\max_i c_i$, put $\bar c=N\mathbf1-c$, and define
$h_k=\sum_{i=1}^k(c_i-a_i-b_i)$ for $0\le k\le n$, as in
\eqref{eq:prefix-heights}.
An \emph{$r$-Hall partition for $(a,b,c)$} is a partition
$\mathcal B$ of $\{1,\ldots,n\}$, together with $r$ distinguished
blocks $B_1,\ldots,B_r$, such that every other block has size one
or two and the following conditions hold.
\begin{enumerate}[label=(\roman*),ref=\roman*]
\item For every $i\in\{1,\ldots,n\}$, we have $c_i-a_i-b_i=1$
if $i\in B_1\cup\cdots\cup B_r$, and $c_i-a_i-b_i=-1$ otherwise.
\item For every $1\le i<j\le n$,
\[
 \operatorname{cmp}\bigl((a_i,b_i,\bar c_i),(a_j,b_j,\bar c_j)\bigr)
 =\begin{cases}
 -1,&\text{if $i$ and $j$ belong to the same block of $\mathcal B$},\\
 1,&\text{if $i\in B_1\cup\cdots\cup B_r$ and $j\notin B_1\cup\cdots\cup B_r$},\\
 0,&\text{otherwise}.
 \end{cases}
\]
\item Every $h_k\ge0$, and the window inequalities
\eqref{eq:window} hold for $a,b,\bar c$.
\end{enumerate}
We call $(a,b,c)$ a \emph{Hall triple} if it admits an $r$-Hall
partition for some integer $1\le r\le n$.
\end{definition}

Changing $N$ adds the same constant to every entry of $\bar c$,
preserving all comparisons and coordinate differences. Thus the
definition of a Hall triple is independent of the choice of $N$.
For the rest of this subsection, fix a Hall triple $(a,b,c)$ and
an $r$-Hall partition $\mathcal B$ for it, with distinguished
blocks $B_1,\ldots,B_r$.

List the coordinate positions outside $B_1\cup\cdots\cup B_r$
as $q_1<\cdots<q_m$. Introduce variables $t_1,\ldots,t_m$,
which will represent $x_{q_1}^{-1},\ldots,x_{q_m}^{-1}$,
respectively, in the coefficient extraction.
Since $\sum_i(c_i-a_i-b_i)=0$, condition~(i) of
Definition~\ref{def:paired-target} gives $m=\sum_{s=1}^r|B_s|$.
Let $\mathcal E$ consist of the pairs $(i,j)$ with $1\le i<j\le m$
for which $\{q_i,q_j\}$ is a two-element block of $\mathcal B$
that is not a distinguished block.

For each $s\in\{1,\ldots,r\}$ and $j\in\{1,\ldots,m\}$,
define $\ell_{s,j}=\#\{i\in B_s:i<q_j\}$, the number of
positions in $B_s$ preceding $q_j$.
Apply the recursion \eqref{eq:canonical-flags} to
$\ell_{s,1},\ldots,\ell_{s,m}$ to obtain the strictly increasing
bounds $f_{s,1},\ldots,f_{s,|B_s|}$.
Define the \emph{Hall polynomial associated with $B_s$} by
\begin{equation}\label{eq:hall-polynomial}
 P_s(t_1,\ldots,t_m)
 =\sum_{\substack{1\le j_1<\cdots<j_{|B_s|}\le m\\
                         j_k\ge f_{s,k}\ (1\le k\le |B_s|)}}
             \prod_{k=1}^{|B_s|}t_{j_k}.
\end{equation}
By Lemma~\ref{lem:hall-flags}, this is precisely the sum over
the Hall-admissible subsets for $\ell_{s,1},\ldots,\ell_{s,m}$.

\begin{proposition}\label{prop:paired-reduction}
Let $(a,b,c)$ be a Hall triple, and fix an $r$-Hall partition
$\mathcal B$ for it. With the notation above, the associated
Hall polynomials satisfy
\begin{equation}\label{eq:paired-coefficient}
 [\mathcal A_c](\kappa_a\kappa_b)
 =[t_1\cdots t_m]\prod_{s=1}^r P_s(t_1,\ldots,t_m)
                  \prod_{(i,j)\in\mathcal E}(1-t_j/t_i).
\end{equation}
\end{proposition}

\begin{proof}
Condition~(iii) of Definition~\ref{def:paired-target} allows us to
apply \eqref{eq:rational-extraction} to the atom coefficient on
the left-hand side of \eqref{eq:paired-coefficient}, giving
\begin{equation}\label{eq:hall-initial-extraction}
 [\mathcal A_c](\kappa_a\kappa_b)
 =[x^{c-a-b}]\prod_{1\le u<v\le n}
 (1-x_u/x_v)^{-\operatorname{cmp}((a_u,b_u,\bar c_u),(a_v,b_v,\bar c_v))}.
\end{equation}

We now substitute $x_{q_j}=t_j^{-1}$ for $1\le j\le m$,
leaving the variables $x_u$ with $u\in B_1\cup\cdots\cup B_r$
unchanged. By condition~(i), the monomial whose coefficient we
are extracting in \eqref{eq:hall-initial-extraction} is
\[
 x^{c-a-b}
 =\left(\prod_{s=1}^r\prod_{u\in B_s}x_u\right)
 \prod_{j=1}^m x_{q_j}^{-1}.
\]
After substitution, we therefore extract the coefficient of
$\left(\prod_{s=1}^r\prod_{u\in B_s}x_u\right)t_1\cdots t_m$.

To rewrite the product on the right-hand side of
\eqref{eq:hall-initial-extraction}, consider the three cases in
condition~(ii).
For $u<v$ in the same distinguished block $B_s$, the comparison
weight is $-1$, giving the factor $1-x_u/x_v$.
For $u\in B_s$ and $u<q_j$, the comparison weight is $1$,
giving $(1-x_u/x_{q_j})^{-1}=(1-x_ut_j)^{-1}$.
Finally, each $(i,j)\in\mathcal E$ corresponds to a two-element
block $\{q_i,q_j\}$, whose factor becomes
$1-x_{q_i}/x_{q_j}=1-t_j/t_i$.
Every other comparison weight is zero, so its factor is $1$.
Combining these factors gives the first equality in
\eqref{eq:hall-factorized-extraction}.
\begin{equation}\label{eq:hall-factorized-extraction}
\begin{aligned}
 [\mathcal A_c](\kappa_a\kappa_b)
 ={}&\left[
 \left(\prod_{s=1}^r\prod_{u\in B_s}x_u\right)t_1\cdots t_m
 \right]
 \prod_{s=1}^r\left(
 \prod_{\substack{u<v\\u,v\in B_s}}(1-x_u/x_v)
 \prod_{j=1}^m\prod_{\substack{u\in B_s\\u<q_j}}
 (1-x_ut_j)^{-1}\right)
 \prod_{(i,j)\in\mathcal E}(1-t_j/t_i)\\[2mm]
 ={}&[t_1\cdots t_m]\,
 \prod_{s=1}^r\left(
 \left[\prod_{u\in B_s}x_u\right]
 \prod_{\substack{u<v\\u,v\in B_s}}(1-x_u/x_v)
 \prod_{j=1}^m\prod_{\substack{u\in B_s\\u<q_j}}
 (1-x_ut_j)^{-1}\right)
 \prod_{(i,j)\in\mathcal E}(1-t_j/t_i).
\end{aligned}
\end{equation}
To justify the second equality in \eqref{eq:hall-factorized-extraction},
note that, for each $s$, the variables $x_u$ with $u\in B_s$ occur
only in the $s$-th parenthesized factor. We can therefore extract
$\prod_{u\in B_s}x_u$ from that factor alone, treating the $t_j$'s
as coefficient variables. Performing these extractions for all $s$,
while leaving the extraction of $t_1\cdots t_m$ outside the product,
gives the stated equality.

We now evaluate the $s$-th parenthesized expression after the
second equality in \eqref{eq:hall-factorized-extraction}.
Fix $s\in\{1,\ldots,r\}$ and write $B_s=\{u_1<\cdots<u_d\}$,
where $d=|B_s|$. By definition, $\ell_{s,j}$ counts the elements
of $B_s$ preceding $q_j$. These elements are exactly
$u_1,\ldots,u_{\ell_{s,j}}$, so the factors involving $t_j$ are
$(1-x_{u_a}t_j)^{-1}$ for $1\le a\le\ell_{s,j}$.
Moreover, since $q_1<\cdots<q_m$, these counts satisfy
$0\le\ell_{s,1}\le\cdots\le\ell_{s,m}\le d$.

We can therefore apply Lemma~\ref{lem:hall-determinant} with its
variables $x_1,\ldots,x_d$ replaced by $x_{u_1},\ldots,x_{u_d}$,
and its integers $\ell_j$ replaced by $\ell_{s,j}$.
This gives the second equality in \eqref{eq:hall-block-extraction};
the first rewrites the expression using the ordered elements of $B_s$.
\begin{equation}\label{eq:hall-block-extraction}
\begin{aligned}
 &\left[\prod_{u\in B_s}x_u\right]
 \prod_{\substack{u<v\\u,v\in B_s}}(1-x_u/x_v)
 \prod_{j=1}^m\prod_{\substack{u\in B_s\\u<q_j}}
 (1-x_ut_j)^{-1}\\[1mm]
 &\quad=[x_{u_1}\cdots x_{u_d}]\,
 \Delta_d(x_{u_1},\ldots,x_{u_d})
 \prod_{j=1}^m\prod_{a=1}^{\ell_{s,j}}(1-x_{u_a}t_j)^{-1}\\[1mm]
 &\quad=\sum_{\substack{1\le j_1<\cdots<j_d\le m\\
             \ell_{s,j_k}\ge k\ (1\le k\le d)}}
 t_{j_1}\cdots t_{j_d}\\[1mm]
 &\quad=P_s(t_1,\ldots,t_m).
\end{aligned}
\end{equation}
For the final equality, Lemma~\ref{lem:hall-flags} identifies the
conditions $\ell_{s,j_k}\ge k$ with $j_k\ge f_{s,k}$, giving
exactly the sum defining $P_s$ in \eqref{eq:hall-polynomial}.
Substituting these identities for all $s$ into
\eqref{eq:hall-factorized-extraction} proves \eqref{eq:paired-coefficient}.
These extractions are finite because the Vandermonde factors are
finite Laurent polynomials, and fixing the total degree in the
variables indexed by each $B_s$ permits only finitely many terms
from the geometric series.
\end{proof}

\section{The counterexample family and its consequences}
\label{sec:negative}

We use the results of Section~\ref{sec:hall-extraction} to compute
the atom coefficient in Theorem~\ref{thm:intro-family}.
For the family of compositions considered here, two distinguished
blocks suffice to give a Hall partition. We obtain an explicit
formula for the coefficient in terms of the parameters of the
family, and determine exactly when it is negative.

\subsection{Evaluating the Hall-polynomial coefficient}

Let $e_h$ denote the elementary symmetric polynomial of degree $h$,
the sum of all squarefree monomials of degree $h$ in its variables.
We use $e_0=1$ and $e_h=0$ if $h<0$ or $h$ exceeds the number
of variables. The two polynomials below will be the Hall polynomials
associated with the two distinguished blocks in the proof of
Theorem~\ref{thm:intro-family}.

\begin{lemma}\label{lem:two-source-count}
Let $p,q$ be positive integers and set $m=p+q-1$.
Let $\mathcal E\subseteq\{1,\ldots,m\}\times\{m+1,\ldots,2m\}$
be a collection of pairs such that every integer in $\{1,\ldots,2m\}$
occurs in exactly one pair. Let $t_1,\ldots,t_{2m}$ be independent
variables. Define
\begin{equation}\label{eq:two-source-polynomials}
\begin{aligned}
 P_1(t)&=\sum_{h=0}^{p}
       e_h(t_1,\ldots,t_m)e_{2p-1-h}(t_{m+1},\ldots,t_{2m}),\\
 P_2(t)&=\sum_{h=0}^{q}
       e_h(t_1,\ldots,t_m)e_{2q-1-h}(t_{m+1},\ldots,t_{2m}).
\end{aligned}
\end{equation}
Then
\begin{equation}\label{eq:two-source-count}
 [t_1\cdots t_{2m}]P_1(t)P_2(t)
       \prod_{(i,j)\in\mathcal E}(1-t_j/t_i)
 =2\binom m{p-1}\binom mp-m\binom{m-1}{p-1}^{\!2}.
\end{equation}
\end{lemma}

\begin{proof}
We compute the coefficient by expanding $P_1$, $P_2$, and each
factor $1-t_j/t_i$, then counting, with signs, the choices of terms
whose product is $t_1\cdots t_{2m}$.

By \eqref{eq:two-source-polynomials}, choosing a monomial from each
of $P_1$ and $P_2$ amounts to choosing subsets
$I,J\subseteq\{1,\ldots,2m\}$ satisfying
\[
\begin{aligned}
 |I|&=2p-1,& |I\cap\{1,\ldots,m\}|&\le p,\\
 |J|&=2q-1,& |J\cap\{1,\ldots,m\}|&\le q.
\end{aligned}
\]
The corresponding monomials are $\prod_{u\in I}t_u$ and
$\prod_{v\in J}t_v$, each with coefficient $1$.
Let $F\subseteq\mathcal E$ consist of the pairs $(i,j)$ for which
we choose $-t_j/t_i$ from the factor $1-t_j/t_i$; from every
other pair's factor, we choose $1$.
The choice of $I,J,F$ contributes the term
\[
 (-1)^{|F|}
 \left(\prod_{u\in I}t_u\right)
 \left(\prod_{v\in J}t_v\right)
 \prod_{(i,j)\in F}\frac{t_j}{t_i}.
\]
For this term to contribute to the coefficient of $t_1\cdots t_{2m}$,
its exponent on every variable must be $1$.
Thus, for each $(i,j)\in\mathcal E$,
\begin{equation}\label{eq:two-source-exponents}
\begin{aligned}
 \mathds{1}_{i\in I}+\mathds{1}_{i\in J}
 &=1+\mathds{1}_{(i,j)\in F},\\
 \mathds{1}_{j\in I}+\mathds{1}_{j\in J}
 &=1-\mathds{1}_{(i,j)\in F}.
\end{aligned}
\end{equation}
Summing the first equality in \eqref{eq:two-source-exponents}
over the pairs gives
\[
 m+|F|=|I\cap\{1,\ldots,m\}|+|J\cap\{1,\ldots,m\}|
 \le p+q=m+1,
\]
so $|F|\le1$.

If $F=\varnothing$, equations \eqref{eq:two-source-exponents}
make $I$ and $J$ complementary subsets of $\{1,\ldots,2m\}$.
Writing $h=|I\cap\{1,\ldots,m\}|$, their bounds give
$m-q=p-1\le h\le p$.
Choose the $h$ elements of $I$ in $\{1,\ldots,m\}$ and its
$2p-1-h$ elements in $\{m+1,\ldots,2m\}$; the complement
then determines $J$ and satisfies its bounds.
Every choice has positive sign, giving the contribution
\[
 \sum_{h=p-1}^{p}\binom mh\binom m{2p-1-h}
 =2\binom m{p-1}\binom mp.
\]

If $F=\{(i,j)\}$, equations \eqref{eq:two-source-exponents}
require $i\in I\cap J$ and $j\notin I\cup J$; every other
index belongs to exactly one of $I,J$.
The numbers of elements of $I$ and $J$ in $\{1,\ldots,m\}$
sum to $m+1=p+q$, so they must equal $p$ and $q$, respectively.
After choosing $(i,j)$ in $m$ ways, choose $p-1$ further elements
of $I$ from $\{1,\ldots,m\}\setminus\{i\}$ and $p-1$ from
$\{m+1,\ldots,2m\}\setminus\{j\}$.
The remaining indices, together with $i$, determine $J$ and
give it $q$ elements in the first group and $q-1$ in the second.
Each choice has negative sign, giving the contribution
\[
 -m\binom{m-1}{p-1}^{\!2}.
\]
Adding the two contributions proves \eqref{eq:two-source-count}.
\end{proof}

\subsection{Proof of the counterexample formula}

\begin{proof}[Proof of Theorem~\ref{thm:intro-family}]
Let $a,b,c$ be the compositions in \eqref{eq:intro-family-arrays}.
We construct a Hall partition for these compositions and identify its
Hall polynomials with those in Lemma~\ref{lem:two-source-count}.
Set $m=p+q-1$, $N=(p+q+2)\delta-1$, and $\bar c=N\mathbf1-c$.

Let $B_1$ be the set of coordinate positions where $a_i=0$, and let
$B_2$ be the set where $b_i=0$. These sets are disjoint and have sizes
$2p-1$ and $2q-1$, respectively. List the remaining positions as
$q_1<\cdots<q_{2m}$. We take $B_1,B_2$ as the distinguished blocks
and pair the remaining positions to form the partition
\[
 \mathcal B=\{B_1,B_2\}\,\cup\,
 \bigl\{\{q_j,q_{2m+1-j}\}:1\le j\le m\bigr\}.
\]
The coordinate triples are constant on each block. From
\eqref{eq:intro-family-arrays}, their values are
\begin{equation}\label{eq:general-family-triples}
 (a_i,b_i,\bar c_i)=
 \begin{cases}
 (0,\delta,(p+q+1)\delta-2),&i\in B_1,\\
 (\delta,0,(p+q+1)\delta-2),&i\in B_2,\\
 ((j+1)\delta,(p+q+1-j)\delta,0),
 &i\in\{q_j,q_{2m+1-j}\},\quad 1\le j\le m.
 \end{cases}
\end{equation}
In coordinate order, there are first $p$ positions of $B_1$, then $q$
positions of $B_2$, followed by $q_1,\ldots,q_m$. Next come the remaining
$p-1$ positions of $B_1$ and $q-1$ positions of $B_2$, followed by
$q_{m+1},\ldots,q_{2m}$.

We verify the conditions of Definition~\ref{def:paired-target}.
All entries of $a,b,c,\bar c$ are nonnegative, and $\max_i c_i=N$.
The triples in \eqref{eq:general-family-triples} give
\[
 c_i-a_i-b_i=N-a_i-b_i-\bar c_i
 =\begin{cases}
 1,&i\in B_1\cup B_2,\\
 -1,&i\notin B_1\cup B_2.
 \end{cases}
\]
There are $2m$ positions of each kind, so $|a|+|b|=|c|$, and
condition~(i) holds.

Within each block of $\mathcal B$, the coordinate triples agree, so their
comparison value is $-1$. Between $B_1$ and $B_2$, one of the first two
entries increases and the other decreases; the same is true between
distinct two-element blocks. These comparisons therefore have value $0$.
Finally, the triple on any two-element block has larger first and second
entries, but a smaller third entry, than the triple on either distinguished
block. The comparison value is consequently $1$ when a distinguished-block
position precedes a position outside $B_1\cup B_2$, and $0$ when their
order is reversed. This proves condition~(ii).

The coordinate order described above shows that the prefix heights
$h_k=\sum_{i=1}^k(c_i-a_i-b_i)$ first increase from $0$ to $p+q$,
then decrease to $1$, increase to $p+q-1$, and finally decrease to $0$.
Thus $0\le h_k\le p+q$. Every strict ascent in $a$ or $b$ has size at
least $\delta$, so its bound in the window inequalities \eqref{eq:window}
is at least $\delta+1>p+q$. Every strict ascent in $\bar c$ has size
$(p+q+1)\delta-2$, giving the bound $(p+q+1)\delta-1>p+q$.
Hence all the window inequalities hold, proving condition~(iii).
The partition $\mathcal B$ is therefore a $2$-Hall partition for $(a,b,c)$.

We now identify the two Hall polynomials. The numbers of positions in
$B_1$ and $B_2$ preceding $q_j$ are
\[
 \ell_{1,j}=\begin{cases}
 p,&1\le j\le m,\\
 2p-1,&m<j\le2m,
 \end{cases}
 \qquad
 \ell_{2,j}=\begin{cases}
 q,&1\le j\le m,\\
 2q-1,&m<j\le2m.
 \end{cases}
\]
A subset $\{j_1<\cdots<j_{2p-1}\}$ is Hall admissible for $B_1$
precisely when $\ell_{1,j_k}\ge k$ for every $k$. For the selected indices
in $\{1,\ldots,m\}$, this requires that their number be at most $p$;
for the remaining selected indices, the inequality holds automatically.
Thus the Hall-admissible subsets for $B_1$ are exactly the subsets of size
$2p-1$ containing at most $p$ elements of $\{1,\ldots,m\}$.
Similarly, those for $B_2$ have size $2q-1$ and contain at most $q$
elements of $\{1,\ldots,m\}$. Their Hall polynomials are therefore
precisely $P_1,P_2$ in \eqref{eq:two-source-polynomials}.

The two-element blocks give
\[
 \mathcal E=\{(j,2m+1-j):1\le j\le m\},
\]
which pairs $\{1,\ldots,m\}$ bijectively with $\{m+1,\ldots,2m\}$.
Applying Proposition~\ref{prop:paired-reduction} and then
Lemma~\ref{lem:two-source-count} yields
\[
\begin{aligned}
 [\mathcal A_c](\kappa_a\kappa_b)
 &=[t_1\cdots t_{2m}]\,P_1(t)P_2(t)
   \prod_{j=1}^{m}(1-t_{2m+1-j}/t_j)\\
 &=2\binom{p+q-1}{p-1}\binom{p+q-1}{p}
       -(p+q-1)\binom{p+q-2}{p-1}^{\!2}.
\end{aligned}
\]
The binomial identities
\[
 \binom{p+q-1}{p-1}=\frac{p+q-1}{q}\binom{p+q-2}{p-1},\qquad
 \binom{p+q-1}{p}=\frac{p+q-1}{p}\binom{p+q-2}{p-1}
\]
give \eqref{eq:intro-family-value}, since
$2(p+q-1)-pq=2-(p-2)(q-2)$.
Its prefactor is positive, so the coefficient is negative exactly when
$(p-2)(q-2)>2$. Finally, the identity
$N(p+q-1,p)=\frac{p+q-1}{pq}\binom{p+q-2}{p-1}^{\!2}$
gives the Narayana expression in \eqref{eq:intro-family-value}.

If this coefficient is negative, then $p,q\ge3$. If also
$p+q\le7$, the positive integers $p-2,q-2$ have sum at most
three and product at most two, a contradiction. At $p+q=8$,
the negative choices are $(3,5),(4,4),(5,3)$, giving respectively
$-105,-350,-105$. Thus the first negative cases within
\eqref{eq:intro-family-arrays} occur in $28$ variables.
\end{proof}

\begin{remark}
The Narayana factor in \eqref{eq:intro-family-value} is identified
algebraically in the proof above. Recall that $N(p+q-1,p)$ counts
Dyck paths of semilength $p+q-1$ with $p$ peaks, or equivalently,
noncrossing partitions of $\{1,\ldots,p+q-1\}$ with $p$ blocks.
It would be instructive to find a direct combinatorial proof
using a natural signed weighting of these objects that explains
the factor $2-(p-2)(q-2)$ and the resulting sign change.
\end{remark}

\subsection{Proofs of the filtration and Lascoux corollaries}
\label{sec:corollary-proofs}

We deduce Corollaries~\ref{cor:intro-filtrations} and~\ref{cor:intro-lascoux}
from the negative atom coefficients in Theorem~\ref{thm:intro-family}.
The implication from Schubert filtrations to atom positivity is already
known; see, for example, Assaf \cite{AssafInsertion}. For completeness,
we explain how such a filtration would express the product character as a
nonnegative sum of Demazure atoms, after specifying the geometric and
character conventions.

\subsubsection{Geometric conventions}\label{sec:geometric-setup}
Fix $n\ge1$, let $G=\operatorname{GL}_n(\mathbb C)$, and let $B\subseteq G$
be the upper-triangular Borel subgroup. We fix the diagonal maximal torus
$T=\{\operatorname{diag}(t_1,\ldots,t_n):t_i\in\mathbb C^\times\}\subseteq B$.
Its character lattice $X^*(T)$ is the group of algebraic group homomorphisms
$T\to\mathbb C^\times$. We identify $X^*(T)$ with $\mathbb Z^n$ by
associating $\nu=(\nu_1,\ldots,\nu_n)$ with the character
$\operatorname{diag}(t_1,\ldots,t_n)\mapsto t_1^{\nu_1}\cdots t_n^{\nu_n}$.
The Weyl group $N_G(T)/T$ is identified with $S_n$, acting on these tuples
by permuting coordinates.

For $\eta\in\mathbb Z^n$, let $\mathbb C_\eta$ be the one-dimensional
$B$-module on which an upper-triangular matrix acts by
the character $\eta$ of its diagonal part. The associated line bundle
$\mathcal L(\eta)=G\times^B\mathbb C_\eta$ on $G/B$ has fibre character
$\eta$ at the point $B/B$. For $\sigma\in S_n$, represented by its
permutation matrix, write $X_\sigma=\overline{B\sigma B/B}\subseteq G/B$
for the Schubert variety and
$\partial X_\sigma=\bigcup_{\tau<\sigma}X_\tau$ for its Schubert boundary,
where $<$ is the strict Bruhat order on $S_n$. We use
$\mathcal L(\eta)$ also for its restriction to these subvarieties.

Given $\nu\in\mathbb Z^n$, let $\eta$ be its weakly increasing
rearrangement and choose the permutation $\sigma\in S_n$ of minimal
Coxeter length satisfying $\sigma\eta=\nu$. The weight $\eta$ is
antidominant with respect to $B$. Following van der Kallen, define the
dual Joseph module $P(\nu)$ and the minimal relative Schubert module
$Q(\nu)$ by
\begin{equation}\label{eq:geometric-modules}
\begin{aligned}
 P(\nu)&=H^0(X_\sigma,\mathcal L(\eta)),\\
 Q(\nu)&=\ker\bigl(H^0(X_\sigma,\mathcal L(\eta))
       \longrightarrow H^0(\partial X_\sigma,\mathcal L(\eta))\bigr),
\end{aligned}
\end{equation}
where $H^0$ denotes the space of global sections and the arrow restricts
sections to the Schubert boundary
\cite[Definitions~2.3.2 and~2.3.4]{VanderKallen1993}.
Both spaces carry the $B$-actions induced by the action on the line bundle.
For a finite-dimensional rational $B$-module $M$, let $M_\mu$ be its
$T$-weight space of weight $\mu\in\mathbb Z^n$. Throughout this subsection
we use the character convention
$\operatorname{ch}M=\sum_{\mu\in\mathbb Z^n}(\dim M_\mu)x^{-\mu}$,
viewed as a Laurent polynomial in $x_1,\ldots,x_n$.
For a weak composition $a$, the section module $P(-a)$ is dual to the
Demazure module with character $\kappa_a$ used in Section~\ref{sec:extraction}.
Our convention of recording weight $\mu$ by $x^{-\mu}$ compensates for
this dualization.

For $1\le i<n$, let $s_i=(i,i+1)\in S_n$ and let
$\mathsf P_i\supset B$ be the minimal parabolic subgroup generated by
$B$ and the permutation matrix of $s_i$.
For a finite-dimensional rational $B$-module $M$, define rank-one induction by
$H_{s_i}(M)=H^0(\mathsf P_i/B,\mathsf P_i\times^B M)$,
where $\mathsf P_i\times^B M$ is the associated vector bundle on
$\mathsf P_i/B$. The action of $\mathsf P_i$ on its global sections is
restricted to $B$.

\begin{proof}[Proof of Corollary~\ref{cor:intro-filtrations}]
We first verify the character identities \eqref{eq:intro-geometric-characters}
in our conventions. For a weakly decreasing composition
$\lambda\in\mathbb Z_{\ge0}^n$, the Schubert variety defining both
$P(-\lambda)$ and $Q(-\lambda)$ is the point $B/B$, so both modules
have character $x^\lambda$.
For a weak composition $u\in\mathbb Z_{\ge0}^n$ with $u_i>u_{i+1}$,
the Demazure recurrence gives
\[
 \operatorname{ch}P(-s_i u)=\pi_i\operatorname{ch}P(-u).
\]
Under the same hypothesis, restriction to the fibre at $B/B$ gives
the evaluation map in the exact sequence
\[
 0\longrightarrow Q(-s_i u)\longrightarrow H_{s_i}Q(-u)
   \longrightarrow Q(-u)\longrightarrow0.
\]
Its character identity is
$\operatorname{ch}Q(-s_i u)=\overline\pi_i\operatorname{ch}Q(-u)$
\cite[Lemma~7.2.3 and Exercise~7.2.4]{VanderKallen1993}.
Starting from the weakly decreasing rearrangement of $u$, these recurrences
agree with the defining recurrences for keys and atoms, proving
\eqref{eq:intro-geometric-characters}.

Now let $a,b,c$ be the compositions of Theorem~\ref{thm:intro-family}
with $(p-2)(q-2)>2$, and set $M=P(-a)\otimes P(-b)$.
If $M$ had a relative Schubert filtration, let $m_\nu$ be the number
of its successive quotients isomorphic to $Q(\nu)$.
The module $Q(\nu)$ contains weight $\nu$
\cite[Remark~2.3.5]{VanderKallen1993}, so every $\nu$ with $m_\nu>0$
is a weight of $M$. Under our character convention, $-\nu$ is therefore
an exponent in the polynomial $\operatorname{ch}M=\kappa_a\kappa_b$.
Thus $-\nu$ is a weak composition of total degree $|a|+|b|$.
Additivity of characters along the filtration and
\eqref{eq:intro-geometric-characters} give
\[
 \kappa_a\kappa_b=\operatorname{ch}M
 =\sum_{\nu}m_\nu\,\operatorname{ch}Q(\nu)
 =\sum_{\nu}m_\nu\,\mathcal A_{-\nu},
 \qquad m_\nu\in\mathbb Z_{\ge0},
\]
where the sums range over the finitely many weights with $m_\nu>0$.
This nonnegative atom expansion contradicts the negative coefficient
in Theorem~\ref{thm:intro-family}.

Every section module over a union of Schubert varieties has a relative
Schubert filtration
\cite[Proposition~2.3.11]{VanderKallen1993}. Refining the layers of a
Schubert filtration in Polo's sense would therefore give the relative
Schubert filtration just excluded.

For $\mathrm{SL}_n$, weights with the same class differ by an integral
multiple of $(1,\ldots,1)$; the fixed total degree chooses a unique lift,
so the same obstruction applies. Since $n=4(p+q-1)$, these failures
occur in type $A_{4(p+q)-5}$. Taking $(p,q)=(3,5)$ gives type $A_{27}$.
\end{proof}

We now prove Corollary~\ref{cor:intro-lascoux} by specializing the
deformation parameter to zero.

\begin{proof}[Proof of Corollary~\ref{cor:intro-lascoux}]
At $\beta=0$, Lascoux polynomials specialize to keys
and Lascoux atoms to Demazure atoms. A nonnegative expansion would
therefore specialize to a nonnegative atom expansion of
$\kappa_a\kappa_b$, contradicting Theorem~\ref{thm:intro-family}.
\end{proof}
\section{Quiver multiplicities and positive density}\label{sec:quiver}

We now use the coefficient-extraction formula \eqref{eq:rational-extraction}
to obtain a positive formula for a family of individual atom coefficients.
Under suitable conditions on the indexing compositions, the extracted
coefficient is the multiplicity of an irreducible representation in the
coordinate ring of a quiver representation space. This interpretation
proves nonnegativity and allows us to describe the coefficient as a count
of lattice points in a polytope. We then show that the resulting criterion
detects a nonvanishing proportion of positive atom coefficients as the
bound on the composition entries grows.

\subsection{The multiplicity formula}\label{subsec:quiver-multiplicity}

\begin{definition}\label{def:quiver-triple}
Let $n\ge1$, and let $a,b,c\in\mathbb Z_{\geq0}^n$ satisfy
$|a|+|b|=|c|$. Choose an integer $N\geq\max_i c_i$, put
$\bar c=N\mathbf1-c$, and define the prefix heights $h_k$ by
\eqref{eq:prefix-heights}.

We call $(a,b,c)$ a \emph{quiver triple} if there is a partition of
$\{1,\ldots,n\}$ into nonempty consecutive intervals $I_1,\ldots,I_s$,
listed in increasing order, satisfying the following conditions.
\begin{enumerate}[label=(\roman*),ref=\roman*]
\item Each of $a,b,\bar c$ is weakly decreasing on every interval $I_p$.
\item For every $1\le p<q\le s$, the comparison weight
\[
 \operatorname{cmp}\bigl((a_i,b_i,\bar c_i),(a_j,b_j,\bar c_j)\bigr)
\]
has the same nonnegative value for all $i\in I_p$ and $j\in I_q$.
\item Every $h_k\ge0$, and the window inequalities \eqref{eq:window}
hold for $a,b,\bar c$.
\end{enumerate}
\end{definition}

These conditions are independent of the choice of $N$, since changing
$N$ adds the same constant to every entry of $\bar c$, preserving all
comparisons and coordinate differences. Fix a quiver triple and an
interval partition $\mathcal I=(I_1,\ldots,I_s)$ satisfying
Definition~\ref{def:quiver-triple}.
The interval partition determines a standard Levi subgroup
$L_{\mathcal I}=\prod_{p=1}^s\operatorname{GL}(V_p)\subseteq\operatorname{GL}_n$,
where $V_p\subseteq\mathbb C^n$ is the coordinate subspace indexed by $I_p$.
We will interpret the extracted atom coefficient as the multiplicity of
an irreducible $L_{\mathcal I}$-module in the coordinate ring of a quiver representation
space.

Define a quiver $Q$, meaning a directed multigraph, with vertices
$1,\ldots,s$ and
$\operatorname{cmp}\bigl((a_i,b_i,\bar c_i),(a_j,b_j,\bar c_j)\bigr)$
distinct arrows $p\to q$ for
$i\in I_p$, $j\in I_q$, and $p<q$. There are zero, one, or two
such arrows. Every arrow points forward, so $Q$ has no directed cycle.
Form the representation
\begin{equation}\label{eq:quiver-algebra}
 \mathcal R_Q=\operatorname{Sym}\!\left(
       \bigoplus_{e:p\to q}V_p\otimes V_q^*\right)
 \quad\text{of }L_{\mathcal I}.
\end{equation}
Thus $\mathcal R_Q$ is the coordinate ring of
$\bigoplus_{e:p\to q}\operatorname{Hom}(V_p,V_q)$, with the
contragredient action on functions. Parallel arrows contribute
distinct summands.

Write $\nu=c-a-b$. For each interval $I_p=\{i_1<\cdots<i_d\}$, set
\[
 \lambda^{(p)}=(\nu_{i_d},\ldots,\nu_{i_1})\in\mathbb Z^d.
\]
Since $a,b,\bar c$ are weakly decreasing on $I_p$, the residual coordinates
$\nu_i=N-a_i-b_i-\bar c_i$ are weakly increasing there. Reversing their
order therefore makes $\lambda^{(p)}$ a dominant weight for
$\operatorname{GL}(V_p)$. Let $V_p^{\lambda^{(p)}}$ denote the irreducible
rational representation of $\operatorname{GL}(V_p)$ with this highest
weight. The entries of $\lambda^{(p)}$ may be negative.

\begin{theorem}\label{thm:quiver-coefficient}
For a quiver triple $(a,b,c)$ and an interval partition $\mathcal I$ satisfying
Definition~\ref{def:quiver-triple}, the associated representations satisfy
\begin{equation}\label{eq:quiver-multiplicity}
 [\mathcal A_c](\kappa_a\kappa_b)
 =\dim\operatorname{Hom}_{L_{\mathcal I}}
       \left(\bigotimes_{p=1}^s V_p^{\lambda^{(p)}},\mathcal R_Q\right).
\end{equation}
Thus the atom coefficient is the multiplicity of the irreducible
$L_{\mathcal I}$-module $\bigotimes_{p=1}^s V_p^{\lambda^{(p)}}$
in $\mathcal R_Q$. In particular, it is nonnegative.
\end{theorem}

\begin{proof}
For each interval $I_p=\{i_1<\cdots<i_d\}$, write
$x_{I_p}=(x_{i_1},\ldots,x_{i_d})$. Every pair $i<j$ in the
coefficient-extraction formula \eqref{eq:rational-extraction} either
lies in a single interval or joins two different intervals. We show
that the factors from the first kind of pair give the Weyl
denominators, while those from the second give the character of
$\mathcal R_Q$.

First, let $i<j$ both lie in $I_p$. Since $a,b,\bar c$ are weakly
decreasing on $I_p$, we have $a_i\ge a_j$, $b_i\ge b_j$, and
$\bar c_i\ge\bar c_j$. The comparison weight is therefore $-1$,
so this pair contributes the factor $1-x_i/x_j$ to
\eqref{eq:rational-extraction}. Multiplying over all pairs within
$I_p$ gives $\Delta_{|I_p|}(x_{I_p})$.

Now let $i\in I_p$ and $j\in I_q$ belong to different intervals.
Since $i<j$, we have $p<q$. By Definition~\ref{def:quiver-triple},
the comparison weight is constant and nonnegative over all such
pairs, and $Q$ has precisely this number of arrows $p\to q$.
Each arrow contributes a summand $V_p\otimes V_q^*$ to
\eqref{eq:quiver-algebra}, with a basis of weight vectors whose
weight monomials are $x_i/x_j$, one for each $i\in I_p$ and
$j\in I_q$. In the symmetric algebra, each basis vector can occur
with any nonnegative exponent. Consequently,
\[
 \operatorname{ch}\mathcal R_Q
 =\prod_{e:p\to q}\ \prod_{\substack{i\in I_p\\j\in I_q}}
 \frac{1}{1-x_i/x_j}.
\]
For each pair of indices in different intervals, the exponent of
its denominator factor is therefore its comparison weight, exactly
as in \eqref{eq:rational-extraction}. Combining these factors with
those from pairs in the same interval gives
\begin{equation}\label{eq:quiver-character-extraction}
 [\mathcal A_c](\kappa_a\kappa_b)
 =[x^\nu]\left(\prod_{p=1}^s\Delta_{|I_p|}(x_{I_p})\right)
          \operatorname{ch}\mathcal R_Q.
\end{equation}

We next determine which homogeneous pieces of $\mathcal R_Q$ can
affect the coefficient in \eqref{eq:quiver-character-extraction}.
For each arrow $e:p\to q$, let $\ell_e\ge0$ be the degree in the
summand $V_p\otimes V_q^*$ of \eqref{eq:quiver-algebra}.
Every monomial of degree $\ell_e$ in this summand has total exponent
$\ell_e$ in the variables indexed by $I_p$, total exponent $-\ell_e$
in those indexed by $I_q$, and zero in the other intervals.

Fix $1\le t<s$, and consider the total exponent in the variables
indexed by $I_1\cup\cdots\cup I_t$. An arrow whose two endpoints
lie among these intervals contributes zero to this total, whereas
an arrow $e:p\to q$ with $p\le t<q$ contributes $\ell_e$.
Each Weyl factor $\Delta_{|I_p|}(x_{I_p})$ preserves the total
exponent in every interval. Thus, for a homogeneous piece of
degrees $(\ell_e)_{e\in Q}$ to contribute to the coefficient of
$x^\nu$ in \eqref{eq:quiver-character-extraction}, its degrees
must satisfy
\begin{equation}\label{eq:quiver-cut-degree}
 \sum_{\substack{e:p\to q\\p\le t<q}}\ell_e
 =\sum_{i=1}^{\max I_t}\nu_i
 =h_{\max I_t}
 \qquad(1\le t<s).
\end{equation}
For every arrow $e:p\to q$, choose an integer $t$ with $p\le t<q$.
Then $\ell_e$ appears in the sum on the left-hand side of
\eqref{eq:quiver-cut-degree}. Since every term in that sum is
nonnegative, $\ell_e\le h_{\max I_t}$. Consequently, only finitely
many choices of $(\ell_e)$ can affect the coefficient in
\eqref{eq:quiver-character-extraction}.
For each fixed tuple $(\ell_e)$, the subspace of $\mathcal R_Q$
consisting of polynomials of degree $\ell_e$ in the variables
associated with each arrow $e:p\to q$ is
\[
 \bigotimes_{e:p\to q}
 \operatorname{Sym}^{\ell_e}(V_p\otimes V_q^*).
\]
The action of $L_{\mathcal I}$ preserves each of these degrees,
so this subspace is an $L_{\mathcal I}$-module. It is finite
dimensional and therefore completely reducible, since
$L_{\mathcal I}$ is reductive.

For a weakly decreasing integer weight $\mu$, denote its rational
Schur character by $s_\mu(z_1,\ldots,z_d)$; determinant twists extend
the usual Schur polynomials to these weights
\cite[Theorem~A2.4(IV)--(V)]{Stanley1999}. Fix a weakly increasing
integer tuple $u=(u_1,\ldots,u_d)$. To compute the coefficient of
$z^u$ in $\Delta_d(z)s_\mu(z)$, multiply the Schur alternant formula
\cite[Chapter~I, \S3, (3.1)]{Macdonald1995} by $\Delta_d(z)$ and
reverse the numerator columns to obtain
\[
 \Delta_d(z)s_\mu(z)
 =\det\!\left(z_i^{\mu_{d+1-j}+j-i}\right)_{i,j=1}^d
 =\sum_{\sigma\in S_d}\operatorname{sgn}(\sigma)
   \prod_{i=1}^d z_i^{\mu_{d+1-\sigma(i)}+\sigma(i)-i}.
\]
The summand indexed by $\sigma$ has exponent vector $u$ exactly when
\[
 u_i+i=\mu_{d+1-\sigma(i)}+\sigma(i)
 \qquad(1\le i\le d).
\]
The sequence $u_1+1,\ldots,u_d+d$ is strictly increasing because
$u$ is weakly increasing. Likewise,
$\mu_d+1,\mu_{d-1}+2,\ldots,\mu_1+d$ is strictly increasing because
$\mu$ is weakly decreasing. The permutation $\sigma$ must therefore
match these two sequences in their increasing order, forcing
$\sigma$ to be the identity. The required equalities then become
$u_i=\mu_{d+1-i}$ for every $i$.

Consequently, a summand contributes precisely when
$\mu=(u_d,\ldots,u_1)$. In that case it is the identity-permutation
summand, whose coefficient is $1$. We obtain
\begin{equation}\label{eq:quiver-weyl-projector}
 [z^u]\Delta_d(z)s_\mu(z)
 =\begin{cases}
 1,&\mu=(u_d,\ldots,u_1),\\
 0,&\text{otherwise},
 \end{cases}
 \qquad u_1\le\cdots\le u_d.
\end{equation}

For each tuple of degrees satisfying \eqref{eq:quiver-cut-degree},
decompose the corresponding finite-dimensional $L_{\mathcal I}$-module
into irreducible representations. Each irreducible has the form
$\bigotimes_{p=1}^s V_p^{\mu^{(p)}}$, with character
$\prod_{p=1}^s s_{\mu^{(p)}}(x_{I_p})$. Write $m_{\boldsymbol\mu}$
for its total multiplicity across these pieces, where
$\boldsymbol\mu=(\mu^{(1)},\ldots,\mu^{(s)})$. There are only
finitely many such pieces and irreducible summands.

The variables belonging to different intervals are disjoint, so
extracting the coefficient of $x^\nu$ from a product of their
characters separates into one coefficient extraction for each interval.
Moreover, $\nu$ is weakly increasing on each $I_p$, and its reversal
there is $\lambda^{(p)}$. Applying \eqref{eq:quiver-weyl-projector}
to each interval therefore gives
\[
\begin{aligned}
 [\mathcal A_c](\kappa_a\kappa_b)
 &=[x^\nu]\left(\prod_{p=1}^s\Delta_{|I_p|}(x_{I_p})\right)
   \operatorname{ch}\mathcal R_Q\\
 &=\sum_{\boldsymbol\mu}m_{\boldsymbol\mu}
   \prod_{p=1}^s\left(
   [x_{I_p}^{\nu|_{I_p}}]\,
   \Delta_{|I_p|}(x_{I_p})s_{\mu^{(p)}}(x_{I_p})
   \right)\\
 &=m_{(\lambda^{(1)},\ldots,\lambda^{(s)})}.
\end{aligned}
\]
Here $\nu|_{I_p}$ denotes the tuple of entries of $\nu$ indexed by
$I_p$, in increasing index order. The first equality is
\eqref{eq:quiver-character-extraction}; the second uses the degree
restriction \eqref{eq:quiver-cut-degree} and the irreducible
decompositions. In the third equality, \eqref{eq:quiver-weyl-projector}
makes each factor zero unless $\mu^{(p)}=\lambda^{(p)}$, in which
case it is one.

Finally, the scalar matrices in $\operatorname{GL}(V_p)$ act on
$V_p^{\lambda^{(p)}}$ with exponent
$\sum_j\lambda_j^{(p)}=\sum_{i\in I_p}\nu_i$.
Thus every copy of $\bigotimes_{p=1}^s V_p^{\lambda^{(p)}}$ in
$\mathcal R_Q$ lies in degrees satisfying \eqref{eq:quiver-cut-degree},
so its multiplicity in the full algebra is exactly the multiplicity
computed above. By Schur's lemma,
\[
 [\mathcal A_c](\kappa_a\kappa_b)
 =m_{(\lambda^{(1)},\ldots,\lambda^{(s)})}
 =\dim\operatorname{Hom}_{L_{\mathcal I}}
 \left(\bigotimes_{p=1}^s V_p^{\lambda^{(p)}},\mathcal R_Q\right).
 \qedhere
\]
\end{proof}

\begin{proof}[Proof of Theorem~\ref{thm:intro-quiver-polytope}]
For a quiver triple $(a,b,c)$, Theorem~\ref{thm:quiver-coefficient}
identifies the atom coefficient with the multiplicity of
$\bigotimes_{p=1}^s V_p^{\lambda^{(p)}}$ in $\mathcal R_Q$,
proving the first equality.

Apply Vergne--Walter's theorem \cite[Theorem~1]{VergneWalter2023}
to the associated acyclic quiver $Q$, with dimension vector
$(|I_1|,\ldots,|I_s|)$ and dominant vertex weights
$(\lambda^{(1)},\ldots,\lambda^{(s)})$. This gives a bounded
rational polytope $P(a,b,c)$ whose integer points count that
multiplicity. Their convention for the action on polynomial functions
agrees with \eqref{eq:quiver-algebra}, as expressed in
\cite[equation~(2.3)]{VergneWalter2023}. Combining their counting
formula with Theorem~\ref{thm:quiver-coefficient} gives
\[
 [\mathcal A_c](\kappa_a\kappa_b)
 =\#\bigl(P(a,b,c)\cap\mathbb Z^M\bigr).
\]
The saturation argument following their Theorem~1 shows that this
multiplicity is positive exactly when $P(a,b,c)$ is nonempty.
For disconnected $Q$, take the Cartesian product of the polytopes
for its connected components. An isolated vertex contributes a point
if its highest weight is zero and the empty set otherwise.

We now prove the algorithmic assertions. To recognize quiver triples,
place an interval boundary between $i$ and $i+1$ exactly when at
least one of $a,b,\bar c$ strictly increases. Every partition
satisfying Definition~\ref{def:quiver-triple} must have these
boundaries, because each composition must be weakly decreasing on
each interval. No other boundary is possible: the adjacent indices
on its two sides would have comparison weight $-1$, contrary to the
required nonnegative comparison weight between different intervals.
Thus these boundaries determine the only candidate partition.
Checking that comparison weights are constant and nonnegative
between each pair of intervals requires $O(n^2)$ comparisons.
The degree equality, nonnegative prefix heights, and window
inequalities \eqref{eq:window} can also be checked in polynomial
time. These checks give a deterministic polynomial-time recognition
algorithm.

For a quiver triple, the resulting quiver has at most $n$ vertices
and $n(n-1)$ arrows. Its vertex weights have $n$ entries altogether,
each obtained by subtracting entries of $a$ and $b$ from an entry
of $c$. The quiver and weights therefore have binary encoding length
polynomial in that of $(a,b,c)$. The construction in
\cite[Theorem~1]{VergneWalter2023} produces the defining inequalities
of $P(a,b,c)$ in deterministic polynomial time.
Their Corollary~2 \cite{VergneWalter2023}, together with saturation,
decides whether the associated multiplicity is positive in
deterministic polynomial time. By
Theorem~\ref{thm:quiver-coefficient}, this decides whether
$[\mathcal A_c](\kappa_a\kappa_b)>0$. All these bounds allow $n$
to vary as part of the input.
\end{proof}

\subsection{A two-source formula and its density}

Theorem~\ref{thm:quiver-coefficient} establishes nonnegativity for
quiver triples. For the interval partition consisting of $\{1,2\}$
and singleton intervals, we now give an explicit criterion for
strict positivity. This criterion will be used to prove the density
statement in Corollary~\ref{cor:intro-quiver-density}.
For a symmetric polynomial $f$, write $[s_\lambda]f$ for the
coefficient of $s_\lambda$ in its Schur-basis expansion.

\begin{proposition}\label{prop:quiver-two-source}
Let $n=m+2$ with $m\ge1$, and assume the hypotheses of
Proposition~\ref{prop:window}. Consider the interval partition
$\mathcal I=(I_1,\ldots,I_{m+1})$ with $I_1=\{1,2\}$ and
$I_{j+1}=\{j+2\}$ for $1\le j\le m$.
If the comparison weights satisfy
\begin{equation}\label{eq:quiver-star-pattern}
 \operatorname{cmp}\bigl((a_i,b_i,\bar c_i),(a_j,b_j,\bar c_j)\bigr)
 =\begin{cases}
 -1,&(i,j)=(1,2),\\
 1,&1\le i\le2<j\le n,\\
 0,&3\le i<j\le n,
 \end{cases}
\end{equation}
and $\nu=c-a-b=(\nu_1,\nu_2,-r_1,\ldots,-r_m)$ with
$0\le\nu_1\le\nu_2$ and $r_j\ge0$, then $(a,b,c)$ is a quiver
triple for $\mathcal I$ and
\begin{equation}\label{eq:quiver-two-source}
 [\mathcal A_c](\kappa_a\kappa_b)
 =[s_{(\nu_2,\nu_1)}]\prod_{j=1}^m h_{r_j}(x_1,x_2).
\end{equation}
Equivalently, this coefficient counts semistandard tableaux of
shape $(\nu_2,\nu_1)$ and content $(r_1,\ldots,r_m)$. In particular,
\begin{equation}\label{eq:quiver-star-support}
 [\mathcal A_c](\kappa_a\kappa_b)>0
 \quad\Longleftrightarrow\quad \max_j r_j\le\nu_2.
\end{equation}
Moreover, the coefficient in \eqref{eq:quiver-two-source} can be
computed by a deterministic algorithm whose running time is polynomial
in $n+|a|+|b|+|c|$.
\end{proposition}

\begin{proof}
The comparison weight $-1$ for the pair $(1,2)$ means that
$a,b,\bar c$ are weakly decreasing on $I_1$. The comparison
weights between different intervals are constant and nonnegative,
so the hypotheses of Proposition~\ref{prop:window} ensure that
$(a,b,c)$ is a quiver triple for $\mathcal I$.
Its quiver has one arrow from vertex $1$ to each vertex $j+1$
for $1\le j\le m$, and no other arrows.
The only nontrivial Weyl factor in
\eqref{eq:quiver-character-extraction} is $1-x_1/x_2$.
For each $j$, the variable $x_{j+2}$ occurs in the character
$\operatorname{ch}\mathcal R_Q$ only through the two factors
associated with that arrow. Expanding these factors as geometric
series gives
\[
 [x_{j+2}^{-r_j}]
 \frac{1}{(1-x_1/x_{j+2})(1-x_2/x_{j+2})}
 =\sum_{t=0}^{r_j}x_1^t x_2^{r_j-t}
 =h_{r_j}(x_1,x_2).
\]
Extracting these coefficients for $1\le j\le m$ in
\eqref{eq:quiver-character-extraction}, and then applying
\eqref{eq:quiver-weyl-projector} with $u=(\nu_1,\nu_2)$, yields
\[
\begin{aligned}
 [\mathcal A_c](\kappa_a\kappa_b)
 &=[x_1^{\nu_1}x_2^{\nu_2}](1-x_1/x_2)
   \prod_{j=1}^m h_{r_j}(x_1,x_2)\\
 &=[s_{(\nu_2,\nu_1)}]\prod_{j=1}^m h_{r_j}(x_1,x_2).
\end{aligned}
\]
This proves \eqref{eq:quiver-two-source}. Iterating the Pieri rule
\cite[Chapter~I, \S5, (5.16)]{Macdonald1995}
identifies this Schur coefficient with the stated tableaux: boxes
added at step $j$ carry label $j$ and form a horizontal strip.

The degree equality gives $\sum_jr_j=\nu_1+\nu_2$. Any one label
occurs at most once in each column, so a tableau requires
$r_j\le\nu_2$ for every $j$. Conversely, assume
$\max_j r_j\le\nu_2$. Sort the prescribed multiset of labels
and put its first $\nu_2$ entries in the top row and its remaining
$\nu_1$ entries in the bottom row. Rows weakly increase. If a
column fails to increase strictly, the entries in positions $i$
and $\nu_2+i$ of the sorted list are equal. This would force
at least $\nu_2+1$ copies of one label, a contradiction.
The empty tableau covers the case
$\nu_1=\nu_2=r_1=\cdots=r_m=0$.

The Pieri recurrence also evaluates \eqref{eq:quiver-two-source}
in polynomial time in $n+|a|+|b|+|c|$. If $M=\sum_jr_j$, at each
step there are at most $M+1$ two-row partitions and at most $M+1$
possible successors per partition. Thus $O(m(M+1)^2)$ integer
additions suffice. Each stored count is bounded by $(m+1)^M$, so
the bit complexity is polynomial as well.
\end{proof}

\begin{proof}[Proof of Corollary~\ref{cor:intro-quiver-density}]
Fix $n\ge4$. We will construct a triple $(a,b,c)$ satisfying
Proposition~\ref{prop:quiver-two-source} for which the atom coefficient
is strictly positive. We will then show that sufficiently small
variations of its entries, preserving $|a|+|b|=|c|$, retain these
properties. Counting integer triples obtained by scaling this family
will give the claimed density.

Put $m=n-2$, set $\delta=10(m+2)$, and choose
$N=(4m+21)\delta+10m$. Define $a,b,\bar c$ by
\begin{equation}\label{eq:quiver-density-seed}
 \begin{aligned}
 (a_1,b_1,\bar c_1)&=((2m+10)\delta+1,(2m+10)\delta+1,\delta+1),\\
 (a_2,b_2,\bar c_2)&=((2m+10)\delta-1,(2m+10)\delta-1,\delta-1),\\
 (a_{j+2},b_{j+2},\bar c_{j+2})&=((3m+14-j)\delta,(3+2j)\delta,(m+5-j)\delta)
              &&(1\le j\le m),
 \end{aligned}
\end{equation}
and set $c_i=N-\bar c_i$. All entries of $a,b,\bar c,c$ lie
strictly between $0$ and $N$. Direct subtraction gives
\[
 c-a-b=(10m-3,\ 10m+3,\ \underbrace{-20,\ldots,-20}_{m\text{ entries}}).
\]
These entries sum to zero, so $|a|+|b|=|c|$. In the notation of
Proposition~\ref{prop:quiver-two-source}, we therefore have
$\nu_1=10m-3$, $\nu_2=10m+3$, and $r_j=20$ for every $j$.
In particular, $0<\nu_1<\nu_2$ and $\max_j r_j<\nu_2$, as
required for strict positivity.

We next verify the comparison weights in \eqref{eq:quiver-star-pattern}.
Each of $a,b,\bar c$ has its first entry strictly larger than its
second. Every entry at a position $3,\ldots,n$ is larger than both
initial entries in $a$ and $\bar c$, and smaller than both initial
entries in $b$. Along positions $3,\ldots,n$, the entries of
$a,\bar c$ strictly decrease, whereas those of $b$ strictly increase.
Thus the comparison weights are $-1$ for $(1,2)$, $1$ for
$i\le2<j$, and $0$ for $3\le i<j$.

The proper prefix heights are
\[
 h_1=10m-3,\qquad h_{j+2}=20(m-j)\quad(0\le j<m).
\]
They are all strictly positive and at most $20m$. Every strict
increase in $a,b,\bar c$ between an initial position and a later
position has size at least $4\delta-1$; every strict increase
between two later positions has size at least $2\delta$.
Since $2\delta>20m$, we have
\begin{equation}\label{eq:quiver-density-margins}
 u_j-u_i>h_k
 \qquad\text{whenever }u\in\{a,b,\bar c\},\ u_i<u_j,
 \ \text{and }i\le k<j.
\end{equation}
These inequalities imply the window inequalities \eqref{eq:window}.
Proposition~\ref{prop:quiver-two-source} now shows that the
constructed triple has a strictly positive atom coefficient.

To obtain a family of such triples, temporarily allow the entries
of $a,b,c$ to be real. We first check that sufficiently small variations
preserving the degree equality retain all the required inequalities.
At the constructed triple, the coordinate comparisons, positivity of
the proper prefix heights, and inequalities
\eqref{eq:quiver-density-margins} are strict, as are $\nu_1>0$,
$\nu_2>\nu_1$, $r_j>0$, and $r_j<\nu_2$. Each therefore has a
positive margin. For example, changing each entry of $a,b,c$ by at most
$\varepsilon$ changes $u_j-u_i-h_k$ by at most $(3k+2)\varepsilon$,
since $\bar c_j-\bar c_i=c_i-c_j$. Taking $\varepsilon$ sufficiently
small preserves its positivity. The same reasoning applies to the other
inequalities, and there are only finitely many, so one sufficiently small
bound preserves them all. In particular, the coordinate comparisons
retain their signs, keeping the comparison weights in
\eqref{eq:quiver-star-pattern} unchanged.

We next check that these conditions survive scaling. Replacing
differences of $\bar c$ by the corresponding differences of $-c$
expresses every inequality above as a homogeneous linear inequality in
$a,b,c$. Multiplying all entries by a positive scalar therefore preserves
each inequality, together with the degree equality. Here we use the
strengthened window inequalities \eqref{eq:quiver-density-margins},
which have no additive constant.

Scale the constructed triple so that all its entries lie strictly
between $0$ and $1$. The degree equality determines the last coordinate by
\begin{equation}\label{eq:quiver-density-balance}
 c_n=\sum_{i=1}^n a_i+\sum_{i=1}^n b_i-\sum_{i=1}^{n-1}c_i.
\end{equation}
Small changes in the other $3n-1$ coordinates produce a small change in
$c_n$. By the preceding small-variation argument, we can therefore choose
a box of positive side lengths in those free coordinates such that every
resulting triple has all entries in $(0,1)$ and satisfies every required
inequality.

Dilate this box by $H$ and choose integer values for its $3n-1$ free
coordinates. Each side length is a positive constant times $H$, so there
are $\Theta_n(H^{3n-1})$ such choices.
Equation~\eqref{eq:quiver-density-balance} makes $c_n$ integral as well.
Every resulting triple lies in $\{0,\ldots,H\}^{3n}$; the scaling
argument preserves the inequalities, so
Proposition~\ref{prop:quiver-two-source} gives a strictly positive atom
coefficient. Thus every such triple belongs to $\mathcal Q_n^+(H)$.

Finally, there are at most $(H+1)^{3n-1}$ triples in
$\{0,\ldots,H\}^{3n}$ satisfying the degree equality: choosing
every coordinate except $c_n$ determines $c_n$ uniquely by
\eqref{eq:quiver-density-balance}. Choices for which this value
lies outside $\{0,\ldots,H\}$ only reduce the count. Together
with the lower bound, this proves
$\#\mathcal Q_n^+(H)=\Theta_n(H^{3n-1})$.
More explicitly, for some constant $\gamma_n>0$ and all sufficiently
large $H$, our family contains at least $\gamma_n H^{3n-1}$ triples.
Its proportion among all triples satisfying the degree equality
is therefore at least
\[
 \frac{\gamma_n H^{3n-1}}{(H+1)^{3n-1}}
 \ge\frac{\gamma_n}{2^{3n-1}}>0.
\]
Every triple with a strictly positive atom coefficient satisfies
the degree equality. Hence the proportion is also bounded away
from zero among all triples with a strictly positive atom coefficient.
\end{proof}

Kouno gives a sufficient condition for a product of two keys to have
a nonnegative expansion in the key basis, and hence in the atom basis
\cite[Theorem~1.3]{Kouno2020}. The pairs $(a,b)$ in the family
constructed above fail Kouno's condition.
\printbibliography
\end{document}